\documentclass{amsart}
\usepackage[margin=3.5cm]{geometry}
\usepackage[dvipsnames]{xcolor}

\usepackage{amsfonts, amsthm, amssymb,verbatim,amscd,amsmath, blindtext}

\usepackage{multirow}
\usepackage{graphicx}
\usepackage{enumitem,yhmath}
\usepackage{amssymb}
\usepackage{ytableau}
\usepackage{tikz}
\usepackage{tikz-cd}
\usepackage{float, pifont}
\usepackage{mathtools}
\usepackage[pagebackref,colorlinks=true,citecolor=blue,linkcolor=BrickRed]{hyperref}
\usepackage{caption} 
\newtheorem{theorem}{Theorem}[section]
\newtheorem{lemma}[theorem]{Lemma}

\newtheorem{proposition}[theorem]{Proposition}
\newtheorem{conjecture}[theorem]{Conjecture}

\theoremstyle{remark}

\theoremstyle{definition}
\newtheorem{definition}[theorem]{Definition}

\DeclareMathOperator{\HS}{HS}
\DeclareMathOperator{\LH}{LH}
\DeclareMathOperator{\CH}{CH}
\DeclareMathOperator{\HH}{H}
\DeclareMathOperator{\ACH}{ACH}
\DeclareMathOperator{\pd}{pd}
\DeclareMathOperator{\supp}{supp}
\DeclareMathOperator{\secondmax}{s-max}
\DeclareMathOperator{\mingens}{Mingens}
\DeclareMathOperator{\lex}{lex}

\title{Edge ideals with linear quotients and without homological linear quotients}

\author[T. Chau]{Trung Chau}
\address[T. Chau]
{Chennai Mathematical Institute, H1 SIPCOT IT Park, Siruseri, Kelambakkam 603103, India.}
\email{chauchitrung1996@gmail.com}

\author[K. K. Das]{Kanoy Kumar Das}
\address[K. K. Das]
{Chennai Mathematical Institute, H1 SIPCOT IT Park, Siruseri, Kelambakkam 603103, India.}
\email{kanoydas0296@gmail.com, kanoydas@cmi.ac.in}

\author[A. Maithani]{Aryaman Maithani}
\address[A. Maithani]
{Department of Mathematics, University of Utah, 155 South 1400 East, Salt Lake City, UT~84112, USA}
\email{maithani@math.utah.edu}

\begin{document}



\begin{abstract}
    A monomial ideal $I$ is said to have homological linear quotients if for each $k\geq 0$, the homological shift ideal $\HS_k(I)$ has linear quotients. It is a well-known fact that if an edge ideal $I(G)$ has homological linear quotients, then $G$ is co-chordal. We construct a family of co-chordal graphs $\{\HH_n^c\}_{n\geq 6}$ and propose a conjecture that an edge ideal $I(G)$ has homological linear quotients if and only if $G$ is co-chordal and $\HH_n^c$-free for any $n\geq 6$. In this paper, we prove one direction of the conjecture. Moreover, we study possible patterns of pairs $(G,k)$ of a co-chordal graph $G$ and integer $k$ such that $\HS_k(I(G))$ has linear quotients. 
\end{abstract}
\maketitle

\section{Introduction}

    Let $S=\Bbbk[x_1,\dots ,x_n]$ be a polynomial ring over a field $\Bbbk$, and $G$ be a finite simple graph with the vertex set $V(G)=\{x_1,\dots, x_n\}$. 
    The \emph{edge ideal} of $G$, denoted by $I(G)$, is the ideal generated by monomials $x_ix_j$ if $x_i$ and $x_j$ form an edge in $G$. 
    The study of edge ideals has been an active area of tremendous interests in commutative algebra ever since its conception by Villareal \cite{Villarreal1990, SVV1994}. 
    Among other topics, the classification of graphs whose edge ideals have an algebraic property is central (cf. \cite{MoreyVillarrealSurvey2012, FranciscoVanTuyl2007, VanTuylVillarreal2008, FrancisciHa2008}). 
    It is a classical result by Fr\"oberg \cite{Froberg} that the edge ideal $I(G)$ has linear resolution if and only if $G$ is co-chordal (we refer to Section~\ref{sec:prem} for the definitions). 
    This result has inspired researchers to find analogs for other classes of ideals such as 
    path ideals \cite{ConcaDeNegri1998, AlFa2015, AlFa2018}, 
    edge ideals of several classes of ``chordal" hypergraphs and simplicial complexes \cite{MR2603461,MR3092695, MR3503700,MR3641976,MR3979275,MR4052306}, 
    or powers of edge ideals \cite{HHZ2004, ForestsCycles, Banerjee2015}, to name a few. 

    Let $I$ be a monomial ideal of $S$, and 
    \[
        \mathcal{F} \colon 
        \cdots \xrightarrow{\partial} F_r\xrightarrow{\partial} F_{r-1} 
        \xrightarrow{\partial} \cdots 
        \xrightarrow{\partial} F_1\xrightarrow{\partial} F_0 \to 0
    \]
    be its graded minimal free resolution over $S$. 
    We can write $F_i = \bigoplus_{j=1}^{b_i} S(-\mathbf{a}_{ij})$, 
    where $S(-\mathbf{a}_{ij})$ denotes $S$ with the shift $\mathbf{a}_{ij}\in \mathbb{N}^n$. 
    For a vector $\mathbf{a} = (a_1,\dots, a_n) \in \mathbb{N}^n$, 
    let $\mathbf{x}^\mathbf{a}$ denote the monomial 
    $x_1^{a_1}\cdots x_n^{a_n}$.
    For any $k \geq 0$, 
    the \emph{$k$-th homological shift ideal} of $I$, 
    denoted by $\HS_k(I)$, is the ideal generated by the monomials 
    $\mathbf{x}^{\mathbf{a}_{kj}}$, $j=1,\dots, b_k$. 
    It is noteworthy that $\HS_0(I)=I$. 
    Homological shift ideals are relatively recent and have become an active subject of study \cite{HMRG20, FH23, HMRZ23, Bayati23, TBR24, CF24}. 
    A monomial ideal $I$ is said to have \emph{homological linear quotients} if $\HS_k(I)$ has linear quotients for all $k\geq 0$. 
    It is of particular interest to find all graphs $G$ such that $I(G)$ has homological linear quotients. 
    By Fr\"oberg's theorem, it is necessary that $G$ is co-chordal. 
    This, however, is not sufficient, as pointed out in \cite[Example 2.3]{FH23}. 
    As such, there have been many attempts to find large classes of such graphs, including forests \cite{FH23}, proper interval graphs \cite{FH23}, and block graphs \cite{TBR24}. 
    Inspired by \cite[Example 2.3]{FH23}, we construct a family of chordal graphs $\{\HH_n\}_{n\geq 6}$ such that none of their complements have homological linear quotients. 
    We propose the following conjecture.

    \begin{conjecture}\label{conj:main}
        An edge ideal $I(G)$ has homological linear quotients if and only if $G$ is co-chordal and $\HH_n^c$-free, i.e., $G$ does not contain $\HH_n^c$ as an induced subgraph, for any $n\geq 6$.
    \end{conjecture}

    We provide a partial answer to this conjecture.

    \begin{theorem}[{Theorem~\ref{thm:conjecture-partial}}]\label{thm:conjecture}
        If the edge ideal $I(G)$ has homological linear quotients, then $G$ is co-chordal and $\HH_n^c$-free for all $n\geq 6$. 
        The converse holds if $G$ has at most 10 vertices.
    \end{theorem}

    We then turn our attention to the possible patterns of $k$ where $\HS_k(I(G))$ fails to have linear quotients, where $G$ is a fixed co-chordal graph. 
    To simplify the notation, we will use $\HS_k(G)$ to denote the homological shift ideal $\HS_k(I(G))$. 
    It is known that if $G$ is co-chordal, then $\HS_0(G)$ and $\HS_1(G)$ have linear quotients (cf. \cite{Froberg} and \cite[Theorem 1.3]{FH23}). 
    Thus, we will only consider $\HS_k(G)$ where $k\geq 2$. 
    We present our main results, where we show that many patterns are possible, though it is noteworthy that not all patterns are, as will be seen in Theorem~\ref{thm:impossible-pattern}.

    \begin{theorem}[{Theorem~\ref{thm:HLQ-LH}}]\label{main-thm-1}
        Let $2\leq s\leq t$ be integers. Then, there exists a graph $G$ such that 
        \begin{enumerate}[label=(\roman*)]
            \item 
            $\HS_{k}(G)$ does not have linear quotients for $k\in [2,s)$;
            \item  
            $\HS_{k}(G)$ has linear quotients for $k\in [s,t]$;
            \item  
            $\HS_{k}(G) = 0$ iff $k\in (t,\infty)$.
        \end{enumerate}
    \end{theorem}

    \begin{theorem}[{Theorem~\ref{thm:HLQ-TH-graph}}]\label{main-thm-2}
        Let $2 \leq s\leq t$ be integers. Then, there exists a graph $G$ such that 
        \begin{enumerate}[label=(\roman*)]
            \item
            $\HS_{k}(G)$ has linear quotients for $k\in [2,s)$;
            \item
            $\HS_{k}(G)$ does not have linear quotients for $k\in [s,t]$;
            \item
            $\HS_{k}(G) = 0$ iff $k\in (t,\infty)$.
        \end{enumerate}
    \end{theorem}

    \begin{theorem}[{Theorem~\ref{thm:HLQ-CH-graph}}]\label{main-thm-3}
        Let $2\leq s\leq t$ be integers. Then, there exists a graph $G$  such that 
        \begin{enumerate}[label=(\roman*)]
            \item
            $\HS_{s}(G)$ does not have linear quotients;
            \item
            $\HS_{k}(G)$ has linear quotients for $k\in [2,s) \cup (s,t]$;
            \item
            $\HS_{k}(G) = 0$ iff $k\in (t,\infty)$.
        \end{enumerate}
    \end{theorem}

    The paper is structured as follows. Section~\ref{sec:prem} provides the background in commutative algebra and graph theory, particularly in the theory of homological shift ideals and the property of having linear quotients. Section~\ref{sec:H-and-LH} introduces the family of graphs $\{\HH_n\}_{n\geq 6}$ and constructs from them a family of graphs that helps us prove Theorem~\ref{main-thm-1}. Section~\ref{sec:ACH-CH} is about proving Theorems~\ref{main-thm-2} and~\ref{main-thm-3}. Finally, in Section~\ref{sec:possible-patterns}, we prove  Theorem~\ref{thm:conjecture} and the possible patterns of integers $k$ and co-chordal graphs $G$ such that $\HS_k(G)$ has linear quotients.

\section*{Acknowledgements}

    The first and second authors acknowledge support from the Infosys Foundation. 
    The third author was supported by NSF grants DMS~2101671 and DMS~2349623. 
    The authors are grateful to the anonymous referee for proofreading the paper carefully.

\section{Preliminaries}\label{sec:prem}

    \subsection{Graph theory terminology}

    We recall some basic graph theory. 
    All graphs in this paper are assumed to be finite simple graphs, i.e., no loops and double edges are allowed. 
    Let $G=(V(G),E(G))$ be a graph where $V(G)=\{z_1,\dots, z_n\}$. 
    The \emph{complement} of $G$, denoted by $G^c$, is the graph with vertex set $V(G^c)=V(G)$ and edge set $E(G^c)=\{z_iz_j\colon z_iz_j\notin E(G)\}$. An \emph{induced subgraph} $H$ of $G$ is a subgraph of $G$ such that $E(H)=\{z_iz_j\colon z_i,z_j\in V(H) \text{ and } z_iz_j\in E(G)\}$. We note that an induced subgraph is uniquely defined by its vertex set.

    For any vertex $z\in V(G)$, a \emph{neighbor} of $z$ in $G$ is a vertex $u$ such that $zu\in E(G)$. The set of all neighbors of $z$ in $G$, denoted by $N_G(z)$. A \emph{complete graph} is a graph such that any two vertices form an edge. Using this terminology, we can now define perfect elimination orderings.

    \begin{definition}
        Let $G$ be a graph with $V(G)=\{z_1,\dots, z_n\}$. An ordering $z_1>\cdots> z_n$ is called a \emph{perfect elimination ordering} in $G$ if $N_{G_i}(z_i)$ induces a complete subgraph of $G$, where $G_i$ is the induced subgraph of $G$ on the vertices $\{z_i,z_{i+1},\dots, z_n\}$.
    \end{definition}

    A graph is said to be \emph{chordal} if it has a perfect elimination ordering. We remark that this is not the original definition of chordal graphs, but rather a classical result proved by Dirac \cite{Dirac1961OnRC}. A graph is said to be \emph{co-chordal} if its complement is chordal.

\subsection{Homological shift ideals}

Let $I$ be a monomial ideal of $S=\Bbbk[x_1,\dots, x_n]$ where $\Bbbk$ is a field. A \emph{(graded) free resolution} of $I$ over $S$ is a graded chain complex of free $S$-modules
\[
\mathcal{F}\colon \cdots \xrightarrow{\partial} F_r\xrightarrow{\partial} F_{r-1} \xrightarrow{\partial} \cdots \xrightarrow{\partial} F_1\xrightarrow{\partial} F_0 \to 0
\]
such that $H_i(\mathcal{F})=0$ for all $i>0$ and $H_0(\mathcal{F})\cong I$. Furthermore, $\mathcal{F}$ is called \emph{minimal} if $\partial(F_i)\subseteq (x_1,\dots ,x_n)F_{i-1}$ all $i\geq 1$. The \emph{projective dimension} of $I$, denoted by $\pd I$, is the length of its minimal free resolution. The ideal $I$ is said to have \emph{linear resolution} if all the matrices representing the differentials of the minimal resolution of $I$ have linear entries. In particular, if $I$ is generated in degree $d$ and has linear resolution, then $\HS_k(I)$ is generated in degree $d+k$.

It is a celebrated theorem of Fr\"oberg that the edge ideal $I(G^c)$ has linear resolution if and only if $G$ is chordal (cf. \cite{Froberg}). In particular, this implies that all the homological shift ideals of $I(G^c)$ are generated in the same degree. We record an explicit description of the generators of these ideals in the following result.

\begin{lemma}[{\cite[Theorem 4.1]{HMRG20}}] \label{lem:gens-HS}
    Let $G$ be the chordal
    graph where  $z_1 > z_2 > \cdots > z_n$ is a perfect elimination ordering for $G$. Set $I=I(G^c)$. Then
    \begin{multline*}
        \HS_k(I) = \big( z_{i_1}z_{i_2}\dots z_{i_{k+2}} \colon 1\leq i_1<i_2<\cdots < i_{k+2} \leq n,\\
        \text{and there exists } t< k+2  \text{ such that } z_{i_t}z_{i_l} \notin E(G) \text{ for each } t<l \leq k+2  \big).
    \end{multline*}
\end{lemma}

\subsection{Linear quotients and linear resolutions}

For any integer $n$, we use $[n]$ to denote the set $\{1,\dots, n\}$. 
Given a monomial ideal $I$, we denote by $\mingens(I)$ the minimal monomial generating set of $I$. 
It is usually difficult to show that a given monomial ideal $I$ has linear resolution directly. Instead, it is more common to show that $I$ has stronger properties that implies having linear resolution. We say that $I$ has \emph{linear quotients} if we can order $\mingens(I)$ as $m_1 > \cdots > m_q$ such that the ideal $(m_1,m_2,\dots, m_{i})\colon m_{i+1}$ is generated by variables for all $i\in [q-1]$. We recall the following well-known result.

\begin{theorem}[{\cite[Theorem 8.2.15]{HerzogHibiBook}}]\label{thm:LQ-implies-LR}
    If a monomial ideal $I$ is generated by monomials of the same degree and has linear quotients, then it has linear resolution.
\end{theorem}

For a monomial ideal $I$ and a monomial $m$, we let $I^{\leq m}$ denote the ideal generated by all minimal monomial generators of $I$ that divide $m$. In particular, $I^{\leq m}$ is a monomial subideal of $I$. Many properties of $I$ can be inherited by $I^{\leq m}$(cf. \cite{BPS98, HHZ04Dirac, CHM24}). Having linear quotients is one of them, provided that $I$ is generated in the same degree.

\begin{lemma}[{\cite[Proposition 2.6]{HMRG20}}]\label{lem:linear-quotients-restriction}
    Let $I$ be a monomial ideal generated by monomials of the same degree. If $I$ has linear quotients, then so does $I^{\leq m}$ for any monomial $m$. 
\end{lemma}

We present an application of this lemma.

\begin{lemma}\label{lem:linear-quotients-induced}
    Let $G$ be a chordal graph and $H$ an induced subgraph. If $\HS_k(I(G^c))$ has linear quotients for some integer $k$, then so does $\HS_k(I(H^c))$.
\end{lemma}
\begin{proof}
    Since $G$ is a chordal graph, so is $H$. By Fr\"oberg theorem \cite{Froberg}, $\HS_k(I(G^c))$ and $\HS_k(I(H^c))$ are generated in degree $k+2$. By \cite[Lemma~4.4]{HHZ04Dirac}, we have 
    \[
    \HS_k(I(H^c)) = \left( \HS_k(I(G^c)) \right)^{\leq \prod_{x\in V(H)}x}.
    \]
    Thus the result follows from Lemma~\ref{lem:linear-quotients-restriction}, as desired.
\end{proof}

Another application of Lemma~\ref{lem:linear-quotients-induced} is that if one can find a monomial $m$ such that $I^{\leq m}$ does not have linear quotients, then neither does $I$. This will be our main tool in proving certain ideals do not have linear quotients, by reducing the problem to ideals that we know do not. The following is one such example.

\begin{lemma}\label{lem:no-LQ-example}
    Let $a,b,c,d\in [n]$ be four distinct integers. Then the ideal $\big( \frac{\prod_{i=1}^n x_i}{x_ax_b}, \ \frac{\prod_{i=1}^n x_i}{x_cx_d} \big)$ does not have linear resolution, and hence also does not have linear quotients.
\end{lemma}
\begin{proof}
    Set $I\coloneqq (m_1,m_2)$, where $m_1\coloneqq \frac{\prod_{i=1}^n x_i}{x_ax_b}$ and $m_2\coloneqq \frac{\prod_{i=1}^n x_i}{x_cx_d}$. A minimal free resolution of $I$ is
    \[
    S\xrightarrow{\begin{pmatrix}
        -x_a x_b\\
        x_c x_d
    \end{pmatrix}} Se_{m_1} \oplus Se_{m_2},
    \]
    where $S=\Bbbk[x_1,\dots, x_n]$. By definition, $I$ does not have linear resolution, and hence also does not have linear quotients by Theorem~\ref{thm:LQ-implies-LR}.  
\end{proof}

To conclude the section, we present our main method in proving that a monomial ideal $I$ has linear quotients with respect to the lex ordering. For a monomial $m=x_{i_1}\cdots x_{i_k}$, let $\supp(m)\coloneqq \{i_1,\dots, i_k\}$ denote the \emph{support} of $m$.

\begin{lemma}\label{lem:LQ-lex}
    Let $x_1>x_2>\cdots >x_{n}$ be an ordering on the variables, and let $(>_{\lex})$ denote the total ordering on monomials using the corresponding lex ordering. Let $I$ be a monomial ideal generated in the same degree. Then $I$ has linear quotients with respect to $(>_{\lex})$ if and only if for any $f,g\in \mingens(I)$ such that $g>_{\lex} f$, there exist integers $i,j$ with $j<i$, $i\in \supp(f)$, and $j\in \supp(g)\setminus \supp (f)$ such that $x_j\frac{f}{x_i} \in \mingens(I)$.
\end{lemma}
\begin{proof}
    By definition, $I$ has linear quotients with respect to $(>_{\lex})$ if and only if for all monomials $f,g\in \mingens(I)$ such that $g>_{\lex} f$, there exists a monomial $f'$ such that
    \begin{equation}\label{condition-LQ}
        f'\in \mingens(I),\;\; 
        f'>_{\lex} f,\;\;
        (f'\colon f) = (x),\;\;
        \text{ and } 
        x\mid g.
    \end{equation}
    Such an $f'$, if it exists, must be of the form $x_j\frac{f}{x_i}$ where $i\in \supp(f)$, $j\in \supp(g)\setminus \supp (f)$, and $j<i$. 
    Conversely, a monomial $f'$ of the form $x_j\frac{f}{x_i}$ with
    $i\in \supp(f)$, $j\in \supp(g)\setminus \supp (f)$, and $j<i$,
    satisfies (\ref{condition-LQ}) if and only if $f'\in \mingens(I)$. This concludes the proof.
\end{proof}

\section{\texorpdfstring{$\HH_n^c$}{HH} and \texorpdfstring{$\LH_{n,r}^c$}{LH}}\label{sec:H-and-LH}

For each $n\geq 6$, let $\HH_n$ be the graph with the vertex set 
\[
V(\HH_n)=\{1,\ldots , n\}
\]
and the edge set 
\begin{align*}
    E(\HH_n) =\, & 
    \{x_ix_{i+1} \colon i\in [n-5]\}
    \cup \{x_{n-4}x_{n-2}, x_{n-2}x_{n-3}\} \\
    &\cup \{x_{n-1}x_j \colon j\in [n]\setminus \{1, n-1\}\} 
    \cup  \{x_nx_j \colon j\in [n]\setminus \{n-3, n\}\}.
\end{align*}

Next, we build up on $\HH_n$ to construct a different family of graphs. For each $n\geq 6$  and $r\geq 0$, let $\LH_{n,r}$ be the graph with the vertex set
\[
V(\LH_{n,r})=V(\HH_n)\cup \{x_{n+1}, \ldots , x_{n+r}\}
\]
and the edge set 
\[
E(\LH_{n,r})=
E(\HH_n) \cup 
\{x_{n+i}x_j\colon i\in [r], j\in [n]\setminus \{n-3\}\}
\cup 
\{x_{n+i}x_{n+j}\colon i,j\in [r]\}.
\]
Pictorially, $\LH_{n,r}^c$ is $\HH_{n}^c$ with $r$ new vertices incident to $x_{n-3}$. 
A vertex is called a \emph{leaf vertex} if it has exactly one neighbor. 
This explains our choice of notation for $\LH_{n,r}^c$, which we call $\HH_n^c$ with leaves. For convenience, we record below the set of edges of $\LH_{n,r}^c$
\[
E(\LH_{n,r}^c)=E(\HH_n^c)\cup \{x_{n-3}x_{n+j}:j\in [r]\}.
\]
We illustrate the first two graphs of these two families, with blue edges representing the ``new" edges needed to construct $\LH_{n,1}$ from $\HH_n$.

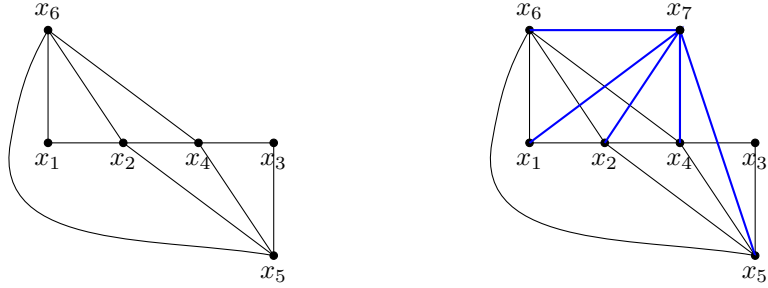
\begin{figure}[h!]
    \centering
        \begin{tikzpicture}[scale=0.5]
            
            \draw [fill] (0,0) circle [radius=0.1];%x1
            \draw [fill] (2,0) circle [radius=0.1];%x2
            \draw [fill] (4,0) circle [radius=0.1];%x4
            \draw [fill] (6,0) circle [radius=0.1];%x3
            \draw [fill] (6,-3) circle [radius=0.1];%x5
            \draw [fill] (0,3) circle [radius=0.1];%x6
            \node at (0,-0.5) {$x_1$};
            \node at (2,-0.5) {$x_2$};
            \node at (4,-0.5) {$x_4$};
            \node at (6,-0.5) {$x_3$};
            \node at (6,-3.5) {$x_5$};
            \node at (0,3.5) {$x_6$};
                        
            \draw (0,0)--(2,0)--(4,0)--(6,0);
            \draw (2,0)--(6,-3)--(4,0);
            \draw (6,0)--(6,-3);
            \draw (0,0)--(0,3)--(2,0);
            \draw (4,0)--(0,3);

             \draw (0, 3) to[out=240, in=80]  (-1, 0) to[out=-100, in=170]  (6, -3);
    \end{tikzpicture}
        \hspace{6em}
    \begin{tikzpicture}[scale=0.5]
            \draw [fill] (0,0) circle [radius=0.1];%x1
            \draw [fill] (2,0) circle [radius=0.1];%x2
            \draw [fill] (4,0) circle [radius=0.1];%x4
            \draw [fill] (6,0) circle [radius=0.1];%x3
            \draw [fill] (6,-3) circle [radius=0.1];%x5
            \draw [fill] (0,3) circle [radius=0.1];%x6

            \draw [fill, black] (4,3) circle [radius=0.1];%x7
            
            \node at (0,-0.5) {$x_1$};
            \node at (2,-0.5) {$x_2$};
            \node at (4,-0.5) {$x_4$};
            \node at (6,-0.5) {$x_3$};
            \node at (6,-3.5) {$x_5$};
            \node at (0,3.5) {$x_6$};
            \node at (4,3.5) {$x_7$};
                        
            \draw (0,0)--(2,0)--(4,0)--(6,0);
            \draw (2,0)--(6,-3)--(4,0);
            \draw (6,0)--(6,-3);
            \draw (0,0)--(0,3)--(2,0);
            \draw (4,0)--(0,3);

            \draw [thick, blue](0,0)--(4,3)--(2,0);
            \draw [thick, blue](4,0)--(4,3);
            \draw [thick, blue](0,3)--(4,3)--(6,-3);
            
            \draw [fill, black] (4,3) circle [radius=0.1];%x7
            
            \draw (0, 3) to[out=240, in=80]  (-1, 0) to[out=-100, in=170]  (6, -3);
    \end{tikzpicture}
        \caption{$\HH_{6}$ and $\LH_{6,1}$}
        \label{fig:H6-LH61}
\end{figure}

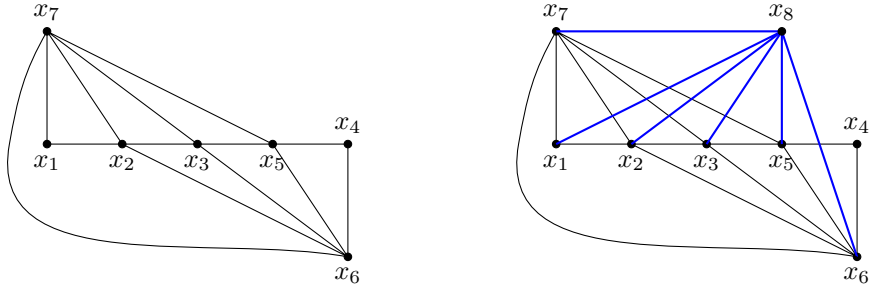
\begin{figure}[h!]
    \centering
        \begin{tikzpicture}[scale=0.5]
            
            \draw [fill] (0,0) circle [radius=0.1];%x1
            \draw [fill] (2,0) circle [radius=0.1];%x2
            \draw [fill] (4,0) circle [radius=0.1];%x3
            \draw [fill] (6,0) circle [radius=0.1];%x5
            \draw [fill] (8,0) circle [radius=0.1];%x4
            \draw [fill] (8,-3) circle [radius=0.1];%x6
            \draw [fill] (0,3) circle [radius=0.1];%x7
            \node at (0,-0.5) {$x_1$};
            \node at (2,-0.5) {$x_2$};
            \node at (4,-0.5) {$x_3$};
            \node at (6,-0.5) {$x_5$};
            \node at (8,0.5) {$x_4$};
            \node at (8,-3.5) {$x_6$};
            \node at (0,3.5) {$x_7$};
                        
            \draw (0,0)--(2,0)--(4,0)--(6,0)--(8,0);
            \draw (2,0)--(8,-3)--(4,0);
            \draw (6,0)--(8,-3)--(8,0);
            \draw (0,0)--(0,3)--(2,0);
            \draw (4,0)--(0,3)--(6,0);

            \draw (0, 3) to[out=240, in=80]  (-1, 0) to[out=-100, in=170]  (8, -3);
    \end{tikzpicture}
        \hspace{4em}
    \begin{tikzpicture}[scale=0.5]
            \draw [fill] (0,0) circle [radius=0.1];%x1
            \draw [fill] (2,0) circle [radius=0.1];%x2
            \draw [fill] (4,0) circle [radius=0.1];%x3
            \draw [fill] (6,0) circle [radius=0.1];%x5
            \draw [fill] (8,0) circle [radius=0.1];%x4
            \draw [fill] (8,-3) circle [radius=0.1];%x6
            \draw [fill] (0,3) circle [radius=0.1];%x7

            \node at (0,-0.5) {$x_1$};
            \node at (2,-0.5) {$x_2$};
            \node at (4,-0.5) {$x_3$};
            \node at (6,-0.5) {$x_5$};
            \node at (8,0.5) {$x_4$};
            \node at (8,-3.5) {$x_6$};
            \node at (0,3.5) {$x_7$};
            \node at (6,3.5) {$x_8$};
                        
            \draw (0,0)--(2,0)--(4,0)--(6,0)--(8,0);
            \draw (2,0)--(8,-3)--(4,0);
            \draw (6,0)--(8,-3)--(8,0);
            \draw (0,0)--(0,3)--(2,0);
            \draw (4,0)--(0,3)--(6,0);

            \draw [thick, blue](0,0)--(6,3)--(2,0);
            \draw [thick, blue](4,0)--(6,3)--(6,0);
            \draw [thick, blue](0,3)--(4,3)--(6,3)--(8,-3);

            \draw [fill, black] (6,3) circle [radius=0.1];%x9
             
            \draw (0, 3) to[out=240, in=80]  (-1, 0) to[out=-100, in=170]  (8, -3);
    \end{tikzpicture}
        \caption{$\HH_{7}$ and $\LH_{7,1}$}
        \label{fig:H7-LH71}
\end{figure}

The goal of this section is to prove the following theorem.
 
\begin{theorem}\label{thm:HLQ-LH}
    For each $n\geq 6$ and $r\geq 1$, we have 
    \begin{enumerate}[label=(\roman*)]
        \item $\HS_k(\LH_{n,r}^c)\neq 0$ does not have linear quotients if $2\leq k< n-4$;
        \item $\HS_k(\LH_{n,r}^c)\neq 0$ has linear quotients if $n-4\leq k \leq n+r-4$;
        \item $\HS_k(\LH_{n,r}^c)=0$ if $n+r-4<k$.
    \end{enumerate}
\end{theorem}

We note that this theorem holds assuming that $r\neq 0$. 
In other words, the above theorem does not yield when $\HS_k(\LH_{n,0}^c)=\HS_k(\HH_{n}^c)$ has linear quotients. 
Despite the similarity between $\HH_n^c$ and $\LH_{n,r}^c$, the information on whether their homological shift ideals have linear quotients are almost polar opposite. We refer to Theorem~\ref{thm:HLQ-H} for a comparison. On the other hand, because of the similarity in structures, we will study some properties of $\HH_n$ in this section, as they will be essential in our study of graphs derived from them.

As the description of homological shift ideals of co-chordal graphs uses perfect elimination ordering of the complement graph, we first show that the natural order of the vertices of $\HH_{n}$ and $\LH_{n,r}$ gives perfect elimination ordering of the corresponding graphs. Since $\HH_{n}$ is an induced subgraph of $\LH_{n,r}$ on the vertex set $[n]$, it is enough to provide a perfect elimination ordering for $\LH_{n,r}$.

\begin{lemma}\label{lem:elm-order-H-and-LH}
    For each $n\geq 6$ and $r\geq 0$, the ordering $x_1>x_2>\cdots >x_{n+r}$ is a perfect elimination ordering for $\LH_{n,r}$. 
    In particular, $\LH_{n,r}$ is a chordal graph. 
    Consequently, for each $n\geq 6$, $\HH_{n}$ is a chordal graph, where $x_1>\cdots >x_n$ gives a perfect elimination ordering for $\HH_n$.
\end{lemma}
\begin{proof}
    Let $G_i$ be the induced subgraph of $\LH_{n,r}$ on the vertices $\{x_i,x_{i+1},\ldots ,x_{n+r}\}$. By definition, it suffices to show that the induced subgraph of $G_i$ on $N_{G_i}(x_i)$ is a complete graph. Indeed, for $i\in [n+r]$, we have
    \begin{align*}
        N_{G_i}(x_i)=\begin{cases}
            \{x_{i+1}\} \cup \{x_{n-1},x_{n},\ldots , x_{n+r}\} &\text{ for } 1\leq i\leq n-5,\\
            \{x_{n-2}\} \cup \{x_{n-1},x_{n},\ldots , x_{n+r}\} &\text{ if } i=n-4, \\
            \{x_{n-2}, x_{n-1}\} &\text{ if } i=n-3, \\
            \{x_{i+1},x_{i+2},\ldots , x_{n+r}\} &\text{ for } n-2\leq i\leq n+r.
        \end{cases}
    \end{align*}
    From the structure of the graph $\LH_{n,r}$, it is straightforward to verify that in each of the above cases, the induced subgraph on $N_{G_i}$ forms a complete subgraph of $G_i$. This completes the proof for $\LH_{n,r}$. The assertion for $\HH_n$ then follows from the fact that $\HH_n$ is precisely $\LH_{n,0}$.
\end{proof}

Next, we study the generators of $\HS_k(\HH_{n}^c)$ and $\HS_k(\LH_{n,r}^c)$ for an integer $k$. Recall that for a co-chordal graph $G$, the ideal $\HS_k(G)$ is generated in degree $k+2$. We first introduce some notation. For a monomial $h=x_{i_1}x_{i_2}\cdots x_{i_{k+2}}$ where $1\leq i_1<i_2<\cdots <i_{k+2}\leq n+r$, set
\begin{align*}
    \min(h) &\coloneqq \min \supp(h) = i_{1},\\
    \max(h) &\coloneqq \max \supp(h) = i_{k+2},\\
    \secondmax(h) &\coloneqq \max(h/x_{\max(h)}) = i_{k+1}.
\end{align*}
Thus, $\secondmax(h)$ denotes the second largest number in its support, explaining our notation.

\begin{definition}\label{def:gens-H-and-LH}
    A monomial $h=x_{i_1}x_{i_2}\cdots x_{i_{k+2}}$ is said to be one of the following \emph{types} if  it satisfies the corresponding condition:
    \begin{enumerate}
        \item[(I)] $\max(h) \leq n-4$ and $[\min(h), \max(h)]\setminus \supp(h) \neq \emptyset$;
        \item[(II)] $\max(h)= n-3$;
        \item[(III)] $\max(h)= n-2$ and either of the following holds:
        \begin{itemize}
            \item $\secondmax(h) \in \{n-3,n-4\}$ and  $[\min(h), \secondmax(h)]\setminus \supp(h) \neq \emptyset$;
            \item $\secondmax(h)\leq n-5$;
        \end{itemize}
        \item[(IV)] $\max(h)= n-1$, $1\in \supp(h)$, and $2\notin \supp(h)$;
        \item[(V)] $\max(h) = n$ and $\secondmax(h)=n-3$.
        \item[(VI)] $\max(h) \geq n+1$ and there is some $q\leq k+1$ such that $i_q=n-3$, $i_{q+1}\geq n$. 
    \end{enumerate}
\end{definition}

\begin{proposition}\label{prop:gens-H}
    For any $k\geq 2$, the ideal $\HS_k(\HH_{n}^c) $ is minimally generated by monomials~$h$, where $\left|\supp(h)\right|=k+2$ and $h$ is of type I, II, III, IV, or V\@.
\end{proposition}
\begin{proof}
    Because of the perfect elimination ordering of $\HH_n$ as given in Lemma~\ref{lem:elm-order-H-and-LH}, we can make use of Lemma~\ref{lem:gens-HS} to conclude our assertion. Thus, it is enough to show that $h=x_{i_1}x_{i_2}\cdots x_{i_{k+2}}$, where $1\leq i_1<i_2<\cdots <i_{k+2}\leq n$,  is a monomial of type (I)--(V)  if and only if there is some $t<k+2$ such that $x_{i_t}x_{i_l}\notin E(\HH_n)$ for all $t< l\leq k+2$. 
    We prove the forward implication by treating each case separately.
    \begin{enumerate}
        \item[(I)] Assume that $\max(h) \leq n-4$ and $[\min(h),\ \max(h)]\setminus \supp(h) \neq \emptyset$. First, we observe that the induced subgraph of $\HH_n$ on the vertices $x_1,x_2, \ldots ,x_{\max(h)}$ is the path graph with edges $\{x_ix_{i+1}\colon 1\leq i\leq \max(h)-1\}$. Since $\min(h)\in \supp(h)$, there exists an index $i_{s}\in [\min(h), \max(h)]\setminus \supp(h)$ such that $i_{s}-1\in \supp(h)$. We now set $i_t=i_{s}-1$. Then $i_{t}+1=i_s\notin \supp(h)$ and hence, $i_{t+1}\geq i_t+2$. Thus, we have $x_{i_t}x_{i_{l}}\notin E(\HH_n)$ for each $t< l\leq k+2$. 

        \item[(II)] Assume that $\max(h)= n-3$. Then $\secondmax(h)\leq n-4$ and note that $x_{\secondmax(h)}x_{\max(h)}\notin E(\HH_n)$. So, if we set $i_t=\secondmax(h)$, then the result follows.

        \item[(III)] Assume that $h$ satisfies one of the conditions given in (III).  First, we observe that the induced subgraph of $\HH_n$ on the vertices $x_1,x_2, \ldots ,x_{\max(h)}=x_{n-2}$ is the path graph with edges $\{x_ix_{i+1}\mid 1\leq i\leq n-5\}\cup\{x_{n-4}x_{x-2}, x_{n-2}x_{x-3}\}$. If $\secondmax(h) \in \{n-3,n-4\}$ and  $[\min(h),\  \secondmax(h)]\setminus \supp(h) \neq \emptyset$. Following similar arguments to those in the first case, we can choose $i_t\in \supp(h)$ such that $i_{t}+1\notin [\min(h),\ \secondmax(h)]\setminus \supp(h)$. Note that $i_t+1\neq n-3$, as otherwise, we would have $i_{t}+1=\secondmax(h)\in \supp(h)$, a contradiction to our choice. Therefore, $i_t+1\leq n-4$ and therefore, $i_t\leq n-5$. Thus, it follows that $N_{\HH_n}(x_{i_t})=\{x_{i_t-1},\ x_{i_t+1}\}$ and since $i_{t+1}\geq i_t+2$, we have $x_{i_t}x_{i_{l}}\notin E(\HH_n)$ for each $t< l\leq k+2$. If $\secondmax(h)\leq n-5$, then we can set $i_t=\secondmax(h)$, as $x_{\secondmax(h)}x_{\max(h)}\notin E(\HH_n)$. The result then follows.

        \item[(IV)] Assume that $\max(h)= n-1$, $1\in \supp(h)$, and $2\notin \supp(h)$. Then, we set $i_t=i_1=1$. Since $N_{\HH_n}(x_{1})=\{x_{2}, x_{n}\}$ and $2,n\notin \supp(h)$, we conclude that $x_{1}x_{i_{l}}\notin E(\HH_n)$ for all $2\leq l\leq k+2$. Thus if we set $i_t=1$, the result then follows.

        \item[(V)] Assume that $\max(h) = n$ and $\secondmax(h)=n-3$. In this case, we set $i_t=\secondmax(h)=n-3$. As $x_{n-3}x_{n}\notin E(\HH_n)$, the result then follows.
    \end{enumerate}

    To prove the converse, let $h$ be a minimal monomial generator of $\HS_k(\HH_n)$. Then by Lemma~\ref{lem:gens-HS}, there exists some $t<k+2$ such that $x_{i_t}x_{i_l}\notin E(\HH_n)$ for all $t< l\leq k+2$. We consider the following cases: 
    \begin{itemize}
        \item[Case 1:] Assume that $\max(h) \leq n-4$. Then $x_{i_t}x_{i_{t+1}}\notin E(\HH_n)$ implies that $i_{t+1}\geq i_t+2$, and hence $i_{t}+1\in [\min(h),\ \max(h)]\setminus \supp(h)$. In particular, this implies that $h$ is of type~I.

        \item[Case 2:] Assume that $\max(h)= n-3$. For this case, it is immediate that $h$ is of type~II.

        \item[Case 3:] Assume that $\max(h)= n-2$. Since $N_{\HH_n}(x_{n-2})=\{n-3,\ n-4\}$, we get $i_t \leq n-5$. Assume that  $\secondmax(h)=n-3$. Then $x_{i_t}x_{i_{t+1}}\notin E(\HH_n)$ implies that $i_{t+1}\neq i_t+1$. Therefore, $i_{t}+1\in [\min(h),\ \secondmax(h)]\setminus \supp(h)$. If $\secondmax(h)=n-4$, it follows that $i_t \leq n-6$. Then again, $i_{t}+1\in [\min(h),\ \secondmax(h)]\setminus \supp(h)$. Finally, excluding the above two cases for $\secondmax(h)$, we must have $\secondmax(h)\leq n-5$. In conclusion, $h$ is of type~III.

        \item[Case 4:] Assume that $\max(h)= n-1$. Note that $N_{\HH_n}(x_{n-1})=\{2,3,\ldots, n-2,n\}$. Thus, $i_t=1=i_1$ and therefore, $i_l\neq 2,n$ for all $2\leq l\leq k+2$, since $N_{\HH_n}(x_{1})=\{2,n\}$. Since $\max(h)= n-1$, this is true exactly when $2\notin \supp(h)$. In other words, $h$ is of type~IV.

        \item[Case 5:] Assume that $\max(h)= n$. Note that $N_{\HH_n}(x_{n})=[n]\setminus \{n-3\}$. Thus, $i_t=n-3$ and since $N_{\HH_n}(x_{n-3})= \{n-1,n-2\}$, it follows that $t=k+1$. Therefore, $\secondmax(h)=n-3$, i.e., $h$ is of type~V. \qedhere
        \end{itemize}
\end{proof}

\begin{proposition}\label{prop:gens-LH}
    For any $k\geq 2$, the ideal $\HS_k(\LH_{n,r}^c)$ is minimally generated by monomials~$h$, where $\left |\supp(h) \right|=k+2$ and $h$ is of type I, II, III, IV, V, or VI\@.
\end{proposition}
\begin{proof}
    Note that the induced subgraph of $\LH_{n,r}$ on the vertex set $\{x_i\colon i\in [n]\}$ is $\HH_n$. Thus, if $\max(h)\leq n$, then it follows from Proposition~\ref{prop:gens-H} that $h\in \HS_k(\LH_{n,r}^c)$ if and only if  $h$ is of type I--V\@. Therefore, we only need to consider the case when $\max(h)>n$. Set $h=x_{i_1}x_{i_2}\cdots x_{i_{k+2}}$, where $1\leq i_1<i_2<\cdots <i_{k+2}\leq n+r$. 
    
    Assume that $\max(h)\geq n+1$, and that $h$ is of type VI\@. As $x_{n-3}x_{n+l}\notin E(\LH_{n,r})$ for all $0\leq l\leq r$, and $\{{i_{q+1}},{i_{q+2}},\ldots ,{\max(h)}\}\subseteq [n,n+r]$, we can set $i_t=i_q$ to obtain ${x_{i_t}x_{i_l}\notin E(\LH_{n,r})}$ for all
    $l \in (t, k+2]$.
    Therefore, $h\in \HS_k(\LH_{n,r}^c)$. Conversely, assume that $h\in \HS_k(\LH_{n,r}^c)$ and that $\max(h)\geq n+1$. By Lemma~\ref{lem:gens-HS}, there exists $t< k+2$ such that $x_{i_t}x_{i_l}\notin E(\LH_{n,r})$ for all $t< l\leq k+2$. Note that $N_{\LH_{n,r}}(x_{\max(h)})=\{x_i\colon i \in [n+r]\setminus \{n-3\}\}$, thus we must have $i_t=n-3$. Since $N_{\LH_{n,r}}(x_{n-3})=\{n-2,n-1\}$, it follows that $i_{t+1}\geq n$. Therefore, $h$ is of type VI, and this completes the proof.
\end{proof}

We are in a position to prove Theorem~\ref{thm:HLQ-LH}. We first determine the nonzero shift ideals of $\LH_{n,r}^c$ for $n \ge 6$ and $r \ge 0$ (in particular, this includes the case $\HH_{n}^{c}$ when $r = 0$).

\begin{proposition}\label{prop:proj-LH}
    For each $n\geq 6$ and $r\geq 0$, we have $\HS_k(\LH_{n,r}^c)\neq 0$ if and only if $k\leq n+r-4$. 
\end{proposition}
\begin{proof}
    It suffices to show that $\HS_{n+r-4}(\LH_{n,r}^c)\neq 0$ and $\HS_{n+r-3}(\LH_{n,r}^c)=0$. Indeed, the monomial $h=\frac{\prod_{i=1}^{n+r}x_i}{x_{n-2}x_{n-1}}$ is of type VI, and thus is a generator of $\HS_{n+r-4}(\LH_{n,r}^c)$ by Proposition~\ref{prop:gens-LH}. In particular, this means that $\HS_{n+r-4}(\LH_{n,r}^c)\neq 0$.
    
    Now, for the sake of contradiction, assume that $\HS_{n+r-3}(\LH_{n,r}^c) \neq 0$, and let $h$ be a minimal monomial generator.
    Then $\left|\supp(h)\right|=n+r-1$, and hence $\max(h)\geq n+r-1$. Then, by Proposition~\ref{prop:gens-LH}, $h$ is either of type V or VI\@. In any case, note that $n-2,n-1\notin \supp(h)$, and this implies that $\left|\supp(h)\right|\leq n+r-2$, a contradiction. 
\end{proof}

For the rest of the section we obtain results for $\LH_{n,r}^{c}$ with $r \geq 1$.
As remarked after Theorem~\ref{thm:HLQ-LH}, we will not obtain information on whether the homological shift ideals of $\HH_n^c$ have linear quotients. 

\begin{proposition}\label{prop:LH-LQ}
    Let $n\geq 6$ and $r\geq 1$ be integers. For each 
    $k \in (n-4, n+r-4]$,
    the ideal $\HS_k(\LH_{n,r}^c)$ has linear~quotients.
\end{proposition}
\begin{proof}
    Let $h\in \HS_k(\LH_{n,r}^c)$ be a minimal monomial generator. 
    By Proposition~\ref{prop:gens-LH}, 
    $\left|\supp(h)\right|=k+2$ and $h$ is of type (I)--(VI). 
    As $\left|\supp(h)\right|=k+2\geq n-1$, 
    we have $\max(h)\geq n-1$. 
    If $\max(h)=n-1$, then we must have $h=\prod_{i=1}^{n-1}x_i$; 
    however, $h$ is then not of type IV, showing $\max(h) \neq n - 1$.
    Similarly, if $\max(h)=n$, then $\secondmax(h)=n-3$, and hence $n-2,n-1\notin \supp(h)$. 
    This implies that $\left|\supp(h)\right|\leq n-2$, a contradiction. 
    Therefore, we must have  $\max(h)\geq n+1$, and hence, $h$ is always of type VI\@. 
    We claim that $\secondmax(h)\neq n-3$. Suppose otherwise that $\secondmax(h)=n-3$. 
    Then $\left|\supp(h)\right|\leq n-3+1=n-2$, a contradiction to the fact that $\left|\supp(h)\right|\geq n-1$. 
    To summarize, $\HS_k(\LH_{n,r}^c)$ is generated by monomials $h$ of type VI where $\left|\supp(h)\right|=k+2$, and $\secondmax(h)\geq n$.

    Now, to show that $\HS_k(\LH_{n,r}^c)$ has linear~quotients, by Lemma~\ref{lem:LQ-lex} and Proposition~\ref{prop:gens-LH}, it suffices to show that for any $f,g$ where 
    \begin{align*}
        &\left|\supp(f)\right|=\left|\supp(g)\right|=k+2,\\
        &\begin{multlined}
            \text{$f$ and $g$ are of type VI, where $\secondmax(f)\geq n$ and $\secondmax(g)\geq n$},
        \end{multlined}
    \end{align*}
    there exist $i,j\in [n+r]$ such that $i>j$, $j\in \supp(g)\setminus \supp(f)$, and $x_j\frac{f}{x_i}$ is also of type VI, and $\secondmax(x_j\frac{f}{x_i})\geq n$. 

    Set
    \begin{align*}
        f&=x_{i_1}x_{i_2}\cdots x_{i_{k+2}},\\
        g&=x_{j_1}x_{j_2}\cdots x_{j_{k+2}}, 
    \end{align*}
    where $1\leq i_1<i_2<\cdots <i_{k+2}\leq n+r$ and $1\leq j_1<j_2<\cdots <j_{k+2}\leq n+r$. Let $s\in [k+2]$ be an index such that $j_1=i_1,\dots, j_{s-1}=i_{s-1},$ and $j_{s}<i_{s}$. Such an $s$ exists since $g>_{\lex} f$, and in particular we have $j_1\leq i_1$. If $\left|\supp(g)\setminus \supp(f)\right|=1$, then $j=j_s$ and $i=i_s$ satisfies the required condition, as in this case we have $x_j\frac{f}{x_i}=g$. Thus, we can assume that $\left|\supp(g)\setminus \supp(f)\right|\geq 2$. In particular, this implies that $s\neq k+2$.
    
    Set $f'\coloneqq x_{j_s}\frac{f}{x_{i_s}}$. It is clear that $f'$ of type VI\@. Thus, it suffices to show that $\secondmax(f')\geq n$. Indeed, we have
    \[
    \secondmax(f')=\begin{cases}
        \secondmax(f) & \text{if $s\leq k$}, \\
        \secondmax(g) & \text{if $s=k+1$},
    \end{cases}
    \]
    which is at least $n$, as desired.
\end{proof}

\begin{proposition}\label{prop:LH-no-LQ}
    Let $n\geq 6$ and $r\geq 1$ be integers. For each $2\leq k\leq n-4$, the ideal $\HS_k(\LH_{n,r}^c)$ does not have linear quotients.
\end{proposition}
\begin{proof}
    As $\LH_{n,r}^c$ contains $\LH_{n,1}^c$ as an induced subgraph for any $r\geq 1$, in view of Lemma~\ref{lem:linear-quotients-induced}, it suffices to show that $\HS_k(\LH_{n,1}^c)$ does not have linear quotients. Set 
    \[m\coloneqq (x_1)(x_3x_4\cdots x_k)(x_{n-3}x_{n-2}\cdots x_{n+1}).\] 
    It now suffices to show that ${\left(\HS_k(\LH_{n,1}^c)\right)}^{\leq m}$ does not have linear quotients.

    Consider a minimal monomial generator 
    $h=x_{i_1}x_{i_2}\cdots x_{i_{k+2}}$ of 
    ${\left(\HS_k(\LH_{n,1}^c)\right)}^{\leq m}$, 
    where $1\leq i_1<i_2<\cdots <i_{k+2}\leq n+1$. 
    Equivalently, $h$ is a generator of $\HS_k(\LH_{n,1}^c)$, 
    and $h \mid m$. Since $\left|\supp(h)\right|=k+2$ and $\left|\supp(m)\right|=k+4$, the number $\max(h)$ must be among the three largest members of $\supp(m)$, i.e., in the set $\{n-1,\ n,\ n+1\}$.
    
    Suppose that $\max(h)=n-1$. 
    Then, $h=(x_1)(x_3x_4\cdots x_k)(x_{n-3}x_{n-2}x_{n-1})=\frac{m}{x_{n}x_{n+1}}$. It follows from Proposition~\ref{prop:gens-LH} that $h$ is a minimal monomial generator of $\HS_k(\LH_{n,1}^c)$.
    
    Suppose that $\max(h)=n$. Since $h$ is a minimal monomial generator of $\HS_k(\LH_{n,1}^c)$, it must be of type V\@. In particular, this means that $n-2,n-1\notin \supp(h)$. Recall that we also have $n+1\notin \supp(h)$. Therefore, $\left|\supp(h)\right|\leq k+1$, a contradiction. 

    Finally, suppose that $\max(f)=n+1$. Since $h$ is a minimal monomial generator of $\HS_k(\LH_{n,1}^c)$, it must be of type VI\@. In particular, this means that $n-2,n-1\notin \supp(h)$. Therefore, we have $h=\frac{m}{x_{n-1}x_{n-2}}$, which is a minimal monomial generator of $\HS_k(\LH_{n,1}^c)$ by Proposition~\ref{prop:gens-LH}. 

    Therefore, $\HS_k(\LH_{n,1}^c)=\left( \frac{m}{x_{n}x_{n+1}},\frac{m}{x_{n-2}x_{n-1}} \right)$, which does not have linear quotients by Lemma~\ref{lem:no-LQ-example}, and this completes the proof.
\end{proof}

\begin{proof}[Proof of Theorem~\ref{thm:HLQ-LH}]
    The theorem follows immediately from Propositions~\ref{prop:proj-LH},~\ref{prop:LH-LQ}, and~\ref{prop:LH-no-LQ}.
\end{proof}
    
\section{\texorpdfstring{$\ACH_{n,r}^c$}{ACH} and \texorpdfstring{$\CH_{n,r}^c$}{CH}} \label{sec:ACH-CH}

Just as in the previous section, let $r\geq 0$ and $n\geq 6$ be two integers. Let $\HH_n$ be the graph with the vertex set
\[
V(\HH_n) = \{x_{r+i} \colon i\in [n] \}
\]
and the edge set
\begin{align*}
    E(\HH_n) =& \{x_{r+i}x_{r+i+1} \colon i\in [n-5]\}\cup \{x_{r+n-4}x_{r+n-2}, x_{r+n-2}x_{r+n-3}\}\\ &
    \cup \{x_{r+n-1}x_{r+j} \colon j\in [n]\setminus \{1, n-1\}\}\cup  \{x_{r+n}x_{r+j} \colon j\in [n]\setminus \{n-3, n\}\}.
\end{align*}

It is clear that this graph is exactly $\HH_n$ from Section~\ref{sec:H-and-LH} with every index increased by $r$. This helps us define the main class of graphs in this section.

Let $\CH_{n,r}$ and $\ACH_{n,r}$ be the graphs with vertex set
\[
V(\CH_{n,r}) = V(\ACH_{n,r}) = V(\HH_n) \cup \{x_1,x_2,\dots, x_{r}\}
\]
and edge sets
\begin{align*}
    E(\CH_{n,r})&= E(\HH_n) \cup \{x_ix_j\colon i,j\in [r], i\neq j \} \cup  \{ x_ix_{r+j} \colon i\in [r] , \ j\in \{n-1,n\}  \},\\
    E(\ACH_{n,r}) &= E(\HH_n) \cup \{x_ix_j\colon i,j\in [r], i\neq j \} \cup  \{ x_ix_{r+j} \colon i\in [r] , \ j\in \{n-2,n-1,n\}  \}\\
    &= E(\CH_{n,r}) \cup  \{ x_ix_{r+n-2} \colon i\in [r] \}.\\
\end{align*}

For convenience, we note down the edge set of their complements:
\begin{align*}
    E(\ACH_{n,r}^c) &= E(\HH_n^c) \cup \{ x_ix_{r+j} \colon i\in [r] , \ j\in [n-3]  \},\\
    E(\CH_{n,r}^c) &= E(\ACH_{n,r}^c) \cup \{ x_ix_{r+n-2} \colon i\in [r]\} .
\end{align*}

It is clear that the induced subgraph of $\ACH_{n,r}$ (or $\CH_{n,r}$) with the vertex set $\{x_{r+j}\colon j\in [n]\}$ is exactly $\HH_n$ with our modified indices above. 
Pictorially, $\ACH_{n,r}^c$ is exactly $\HH_{n}^c$ with $r$ new vertices incident to $x_{r+j}$ for all $j\in [n-3]$, 
while $\CH_{n,r}^c$ is exactly $\HH_{n}^c$ with $r$ new vertices incident to $x_{r+j}$ for any $j\in [n-2]$, i.e., all vertices of $\CH_n$ except for its two leaves. 
Due to these facts, we use the terms \emph{$\HH_{n}^c$ with cones} and \emph{$\HH_{n}^c$ with almost-cones} for $\CH_{n,r}^c$ and $\ACH_{n,r}^c$, respectively.

We illustrate the two graphs $\CH_{n,1}$ and $\ACH_{n,1}^c$  below.

\begin{figure}[h!]
    \centering
        \begin{tikzpicture}[scale=0.5]

            \draw [fill] (0,0) circle [radius=0.1];%x2
            \draw [fill] (2,0) circle [radius=0.1];%x3
            \draw [fill] (4,0) circle [radius=0.1];%x4
            \draw [fill] (6,0) circle [radius=0.1];%x6
            \draw [fill] (8,0) circle [radius=0.1];%x5
            \draw [fill] (8,-3) circle [radius=0.1];%x7
            \draw [fill] (0,3) circle [radius=0.1];%x8

            \node at (6,3.5) {$x_1$};
            \node at (0,-0.5) {$x_2$};
            \node at (2,-0.5) {$x_3$};
            \node at (4,-0.5) {$x_4$};
            \node at (6,-0.5) {$x_6$};
            \node at (8,0.5) {$x_5$};
            \node at (8,-3.5) {$x_7$};
            \node at (0,3.5) {$x_8$};
                        
            \draw (0,0)--(2,0)--(4,0)--(6,0)--(8,0);
            \draw (2,0)--(8,-3)--(4,0);
            \draw (6,0)--(8,-3)--(8,0);
            \draw (0,0)--(0,3)--(2,0);
            \draw (4,0)--(0,3)--(6,0);
            \draw [thick,blue] (0,3)--(6,3)--(8,-3);

            \draw [fill, black] (6,3) circle [radius=0.1];%x1

            \draw  (0, 3) to[out=240, in=80]  (-1, 0) to[out=-100, in=170]  (8, -3);
    \end{tikzpicture}
        \hspace{4em}
    \begin{tikzpicture}[scale=0.5]
    
            \draw [fill] (0,0) circle [radius=0.1];%x1
            \draw [fill] (2,0) circle [radius=0.1];%x2
            \draw [fill] (4,0) circle [radius=0.1];%x3
            \draw [fill] (6,0) circle [radius=0.1];%x5
            \draw [fill] (8,0) circle [radius=0.1];%x4
            \draw [fill] (8,-3) circle [radius=0.1];%x6
            \draw [fill] (0,3) circle [radius=0.1];%x7

            \node at (0,-0.5) {$x_2$};
            \node at (2,-0.5) {$x_3$};
            \node at (4,-0.5) {$x_4$};
            \node at (6,-0.5) {$x_6$};
            \node at (8,0.5) {$x_5$};
            \node at (8,-3.5) {$x_7$};
            \node at (0,3.5) {$x_8$};
            \node at (6,3.5) {$x_1$};
                        
            \draw (0,0)--(2,0)--(4,0)--(6,0)--(8,0);
            \draw  (2,0)--(8,-3)--(4,0);
            \draw (6,0)--(8,-3)--(8,0);
            \draw (0,0)--(0,3)--(2,0);
            \draw (4,0)--(0,3)--(6,0);

            \draw [thick,blue] (0,3)--(6,3)--(8,-3);
            \draw [thick,blue] (6,0)--(6,3);

            \draw [fill, black] (6,3) circle [radius=0.1];%x1

            \draw (0, 3) to[out=240, in=80]  (-1, 0) to[out=-100, in=170]  (8, -3);
    \end{tikzpicture}
        \caption{$\CH_{7,1}$ and $\ACH_{7,1}$}
        \label{fig:CH71-ACH71}
\end{figure}
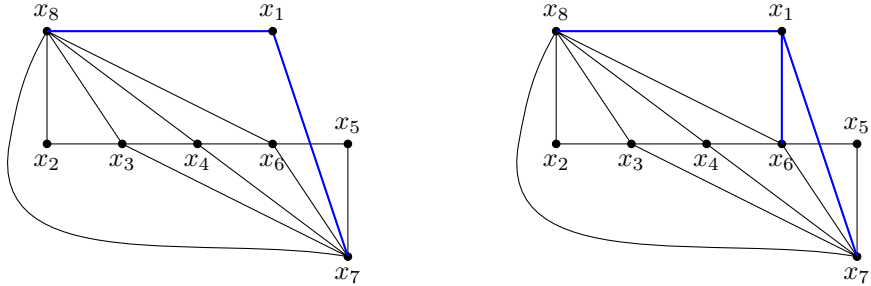

The goal of this section is to prove the following two theorems regarding which homological shift ideals of these two graphs have linear quotients.

\begin{theorem}\label{thm:HLQ-TH-graph}
    For each $n\geq 6$ and $r\geq 0$, we have
    \begin{enumerate}[label=(\roman*)]
        \item $\HS_k(\ACH_{n,r}^c)\neq 0$ has linear quotients if $2\leq k < n-4$;
        \item $\HS_k(\ACH_{n,r}^c)\neq 0$ does not have linear quotients if $n-4\leq k \leq n+r-4$;
        \item $\HS_k(\ACH_{n,r}^c)=0$ if $n+r-4<k$.
    \end{enumerate}
\end{theorem}

\begin{theorem}\label{thm:HLQ-CH-graph}
    For each $n\geq 6$ and $r\geq 0$, we have
    \begin{enumerate}[label=(\roman*)]
        \item $\HS_k(\CH_{n,r}^c)\neq 0$ has linear quotients if $2\leq k < n-4$ or $n-4< k \leq n+r-4$;
        \item $\HS_{n-4}(\CH_{n,r}^c)\neq 0$ does not have linear quotients;
        \item $\HS_k(\CH_{n,r}^c)=0$ if $n+r-4<k$.
    \end{enumerate}
\end{theorem}

We first provide a perfect elimination ordering for $\CH_{n,r}$ and $\ACH_{n,r}$. Due to their similarity, it should come as no surprise that the natural ordering works for both.

\begin{lemma}\label{lem:elimination-ordering-CH-TH}
    For each $n\geq 6$ and $r\geq 0$, the ordering $x_1> x_2> \cdots > x_{n+r-1} >x_{n+r}$ is a perfect elimination ordering for both $\CH_{n,r}$ and $\ACH_{n,r}$. In particular, $\CH_{n,r}$ and $\ACH_{n,r}$ are  chordal graphs.
\end{lemma}

\begin{proof}
    The proofs are similar for the two graphs. Let $G$ be either $\CH_{n,r}$ or $\ACH_{n,r}$, and $G_i$ be the induced subgraph of $G$ on the vertices $\{x_i, x_{i+1},\dots, x_{n+r}\}$. By definition, it suffices to show that $N_{G_i}(x_i)$ induces a complete subgraph of $G$. 
    
    Indeed, if $i\in [r]$, then one can verify that 
    \[
    N_{G_i}(x_i)= \begin{cases}
        \{x_{i+1}, x_{i+2}, \cdots, x_{r}\} \cup \{x_{r+n-1},x_{r+n}\} &\text{if $G=\CH_{n,r}$,}\\
        \{x_{i+1}, x_{i+2}, \cdots, x_{r}\} \cup \{x_{r+n-2},x_{r+n-1},x_{r+n}\} &\text{if $G=\ACH_{n,r}$},
    \end{cases}
    \]
    induces a complete subgraph of $G$ by definition. On the other hand, if $i>r$, then it suffices to assume that $G$ is now the induced subgraph of  with the vertex set $\{x_{r+i}\colon i\in [n]\}$, i.e., the graph $\HH_n$. The result then follows from Lemma~\ref{lem:elm-order-H-and-LH}, as desired.
\end{proof}

For the rest of the section, we fix $2\leq k\leq r+n-4$, and inspect the ideals $\HS_k(\CH_{n,r}^c) $  and $\HS_k(\ACH_{n,r}^c)$. First we introduce some notation. For each monomial $h=x_{i_1}x_{i_2}\cdots x_{i_{k+2}}$ where $1\leq i_1<i_2<\cdots <i_{k+2}\leq r+n$, set
\begin{align*}
    Q_h&\coloneqq \supp(h) \cap \{r+j \colon j\in [n] \},\\
    \max(h) &\coloneqq \max \supp(h) = i_{k+2},\\
    \secondmax(h) &\coloneqq \max(h/x_{\max(h)}) = i_{k+1}.
\end{align*}
Just like the minimal monomial generators of $\HS_k(\HH_n^c)$ in Proposition~\ref{prop:gens-H}, those of $\HS_k(\CH_{n,r}^c) $  and $\HS_k(\ACH_{n,r}^c)$ also follow certain rules. As the two graphs are similar, the same applies to these generators. Therefore, we define the following types as before.

\begin{definition}
    A monomial $h$ is said to be one of the following \emph{type} if  it satisfies the corresponding condition:
    \begin{enumerate}
        \item[(1)] $\max(h)\in [r+1,\ r+n-4]$ and either $\supp(h) \cap [r] \neq \emptyset$ or $[\min Q_h,\  \max(h)]\setminus Q_h \neq \emptyset$;
        \item[(2)] $\max(h)=r+n-3$;
        \item[(3A)] $\max(h)=r+n-2$ and $\secondmax(h)\in \{r+n-3,r+n-4\}$ and $ [\min Q_h, \secondmax(h)] \setminus Q_h \neq \emptyset$;
        \item[(3B)] $\max(h)=r+n-2$ and  $\secondmax(h) \in [r+1,r+n-5]$;
        \item[(3C)] $\max(h)=r+n-2$ and $\supp(h) \cap [r] \neq \emptyset$;
        \item[(4)] $\max(h)=r+n-1$, and we have $r+1\in \supp(h)$, and $r+2\notin \supp(h)$;
        \item[(5)] $\max(h)=r+n$ and $\secondmax(h)=r+n-3$.
    \end{enumerate}
\end{definition}

We remark that these types are mutually exclusive, the types (except for 3C) resemble the conditions in Proposition~\ref{prop:gens-H}, and if $h$ is of one of the types, then $Q_h\neq \emptyset$. Now we are ready to characterize the minimal monomial generators of $\HS_k(\ACH_{n,r}^c) $  and $\HS_k(\CH_{n,r}^c)$.

\begin{proposition}\label{prop:gens-TH}
    For $k\geq 2$, the ideal $\HS_k(\ACH_{n,r}^c) $ is minimally generated by monomials~$h$, where $\left|\supp(h)\right|=k+2$ and $h$ is of type 1, 2, 3A, 3B, 4, or 5.
\end{proposition}

\begin{proof}
    Since $\ACH_{n,r}$ is chordal (Lemma~\ref{lem:elimination-ordering-CH-TH}),  the ideal $\HS_k(\ACH_{n,r}^c) $ is minimally generated by monomials of degree $k+2$. Now let $h=x_{i_1}x_{i_2}\cdots x_{i_{k+2}}$ where $1\leq i_1<i_2<\cdots < i_{k+2}\leq r+n$. Then clearly  $i_{k+2}=\max(h)$ and  $i_{k+1}= \secondmax(h)$. 

    By Lemma~\ref{lem:gens-HS}, it suffices to show that $h$ is of type 1, 2, 3A, 3B, 4, or 5 if and only if there exists $t<k+2$ such that $x_{i_t}x_{i_l}\notin E(\ACH_{n,r})$ for all $t<l\leq k+2$. 
    We note that  the induced subgraph of $\ACH_{n,r}$ with the vertex set $\{x_{r+i}\colon i\in [n]\}$ is exactly $\HH_n$, and $Q_h$ is precisely the set of vertices of $\HH_n$ that divide $h$. 
    Therefore, if $i_{1}> r$, then $Q_h\neq \emptyset$ vacuously and thus the result follows from Proposition~\ref{prop:gens-H}. Now we can assume that $i_{1}\leq r$, i.e., $\supp(h)\cap [r]\neq \emptyset$.

    Suppose that $h$ is of type 1, 2, 3A, 3B, 4, or 5. We have the following cases:
    \begin{itemize}
        \item Suppose that $h$ is of type 1, i.e., $i_{k+2}\in [r+1,\ r+n-4]$. Let $t<k+2$ be the index such that 
        $i_t= \max \left(\supp(h)\cap [r]\right)$. 
        Then, if $j\in \supp(h)$ and $j>i_t$, we have $j>r$. 
        Hence, for any $t<l\leq k+2$, we have $j_l>r$, and thus, $x_{i_t}x_{i_l}\notin E(\ACH_{n,r})$, as desired. 
        \item Suppose that $h$ is of type 2, i.e., $i_{k+2}=r+n-3$. Since $x_ix_{r+n-3}\notin E(\ACH_{n,r})$ for any $i<r+n-3$, the index $t=k+1$ satisfies the condition $x_{i_t}x_{i_l}\notin E(\ACH_{n,r})$ for any $t<l\leq k+2$, as desired.
        \item Suppose that $h$ is of type 3A, 3B, 4, or 5. 
        Then $i_{k+2}\in \{r+n-2,\ r+n-1,\ r+n\}$, and, in particular, $x_ix_{i_{k+2}}\in E(\ACH_{n,r})$ for all $i\in [r]$. 
        Thus, if there exists $t<k+2$ such that 
        $x_{i_t}x_{i_l}\notin E(\ACH_{n,r})$ for all $t<l\leq k+2$, we must have $i_t>r$. 
        In other words, $x_{i_t}$ and any $x_{i_l}$ where $l>t$, are vertices of $\HH_n$. 
        The existence of the required $t$ then follows from the proof of Proposition~\ref{prop:gens-H}, as desired.
    \end{itemize}

    Conversely, suppose that there exists $t<k+2$ such that $x_{i_t}x_{i_l}\in E(\ACH^c_{n,r})$ for all $t<l\leq k+2$. 
    If $i_t>r$, then the indices $\{i_1,\dots, i_{t-1}\}$ are inconsequential. 
    Hence, from the proof of Proposition~\ref{prop:gens-H}, we conclude that $h$ is of type 1, 2, 3A, 3B, 4, or 5. 
    Thus, we can assume that $i_t\in [r]$. 
    Since $x_{i_t}x_{i_{k+2}}\notin E(\ACH_{n,r})$, 
    we have $i_{k+2}=r+j$ for some $j\in [n-3]$. 
    If $i_{k+2}=r+n-3$, then $h$ is of type 2, as desired. 
    On the other hand, if $i_{k+2}=r+j$ for some $j\in [n-4]$, then $h$ is of type 1, as desired.
\end{proof}

\begin{proposition}\label{prop:gens-CH}
    For any $k\geq 2$, the ideal $\HS_k(\CH_{n,r}^c) $ is minimally generated by monomials~$h$, where $\left|\supp(h)\right|=k+2$,  $Q_h\neq \emptyset$, and $h$ is of type 1, 2, 3A, 3B, 3C, 4, or 5.
\end{proposition}

\begin{proof}
    Since $\CH_{n,r}$ is chordal (Lemma~\ref{lem:elimination-ordering-CH-TH}),  the ideal $\HS_k(\CH_{n,r}^c) $ is minimally generated by monomials of degree $k+2$. Now let $h=x_{i_1}x_{i_2}\cdots x_{i_{k+2}}$ where $1\leq i_1<i_2<\cdots < i_{k+2}\leq r+n$. Then clearly  $i_{k+2}=\max(h)$ and  $i_{k+1}= \secondmax(h)$.
    
    By Lemma~\ref{lem:gens-HS}, it suffices to show that $h$ is of type 1, 2, 3A, 3B, 3C, 4, or 5 if and only if there exists $t<k+2$ such that $x_{i_t}x_{i_l}\notin E(\ACH_{n,r})$ for any $t<l\leq k+2$. We note that even the perfect elimination orderings for the two graphs $\CH_{n,r}$ and $\ACH_{n,r}$ are the same from Lemma~\ref{lem:elimination-ordering-CH-TH};
    the only difference is their edge sets: $E(\CH_{n,r}^c) = E(\ACH_{n,r}^c) \cup \{x_i x_{r+n-2} \colon i\in [r] \}$. 
    Thus, the result follows immediately from Proposition~\ref{prop:gens-TH}, if $\max(h)< r+n-2$. On the other hand, if $i_{k+2}=\max(h)\in \{ r+n-1,\ r+n\}$, then $i_t\notin [r]$ since $x_ix_{i_{k+2}} \notin E(\CH_{n,r}^c)$. In particular, this means that the existence of the extra edges $ \{x_i x_{r+n-2} \colon i\in [r] \}$ does not affect 
    the existence of $t<k+2$ such that $x_{i_t}x_{i_l}\notin E(\ACH_{n,r})$ for any $t<l\leq k+2$. The result then follows immediately from Proposition~\ref{prop:gens-TH}. Therefore, we now assume that $\max(h)\geq r+n-2$.
    
    Just as with the graph $\ACH_{n,r}$, we note that  the induced subgraph of $\CH_{n,r}$ with the vertex set $\{x_{r+i}\colon i\in [n]\}$ is exactly $\HH_n$, and $Q_h$ is precisely the set of vertices of $\HH_n$ that divide $h$. Therefore, if $i_{1}> r$ then $Q_h=\supp(h)$, and thus from the proof of Proposition~\ref{prop:gens-H}, we conclude that $h$ is of type 3A or 3B, as desired. Now we can assume that $i_{1}\leq r$, i.e., $\supp(h)\cap [r]\neq \emptyset$. This means that $h$ is of type 3C. Therefore, it now suffices to show that there exists $t<k+2$ such that $x_{i_t}x_{i_l}\notin E(\CH_{n,r})$ for any $t<l\leq k+2$. Indeed, set $t$ to be the index such that $i_t\coloneqq \max \left( \supp(h)\cap [r] \right)$. Then $x_{i_t}x_{r+j}\notin E(\CH_{n,r})$ for any $j\in [n-2]$ by definition, and thus the result follows.
\end{proof}

We next determine when $\HS_k(\CH_{n,r}^c)$ or $\HS_k(\ACH_{n,r}^c)$ is nonzero.

\begin{proposition}\label{prop:proj-CH-TH}
    For any $n\geq 6$ and $r\geq 0$, the ideal $\HS_k(\CH_{n,r}^c)$ is nonzero if and only if $\HS_k(\ACH_{n,r}^c)$ is nonzero if and only if $k\leq r+n-4$.
\end{proposition}

\begin{proof}
    We will prove the statement for $\CH_{n,r}^c$, and remark that similar arguments follow for $\ACH_{n,r}^c$. It suffices to show that $\HS_{n+r-4}(\CH_{n,r}^c)\neq 0$ and $\HS_{n+r-3}(\CH_{n,r}^c)= 0$. Indeed, the monomial $\frac{\prod_{i=1}^{r+n}x_j}{x_{r+n-2}x_{r+n-1}}$ is of type 5, and thus is a generator of $\HS_{n+r-4}(\CH_{n,r}^c)$ by Proposition~\ref{prop:gens-CH}. In particular, this implies that $\HS_{n+r-4}(\CH_{n,r}^c)\neq 0$.

    Now, for the sake of contradiction, assume that 
    $\HS_{n+r-3}(\CH_{n,r}^c) \neq 0$, and let $h$ be a minimal monomial generator.
    Then $\left|\supp(h)\right|=r+n-1$, and thus, $h$ is of either type 4 or 5 by Proposition~\ref{prop:gens-CH}. 
    This results in a contradiction as the support of a monomial of either type is at most $r+n-2$ by definition. 
\end{proof}

We consider the ordering $x_1>x_2>\cdots>x_{r+n-1}>x_{r+n}$ and use $>_{\lex}$ to denote the total ordering on monomials using the corresponding lex ordering. We will show that $\HS_k(\CH_{n,r}^c)$ and $\HS_k(\ACH_{n,r}^c)$ have linear quotients (with appropriate $k$) with respect to this lex ordering. By Lemma~\ref{lem:LQ-lex}, we need to show that for any two minimal monomial generators $f,g$ such that $g>_{\lex} f$ and $\left|\supp(g)\setminus \supp(f)\right|\geq 2$, there exists $i\in \supp(f)$ and $j\in \supp(g)\setminus \supp (f)$ such that and $i>j$ and $x_j\frac{f}{x_i} \in \mingens(I)$.

Set
\begin{align*}
    f&=x_{i_1}x_{i_2}\cdots x_{i_{k+2}},\\
    g&=x_{j_1}x_{j_2}\cdots x_{j_{k+2}}, 
\end{align*}
where $1\leq i_1<i_2<\cdots <i_{k+2}\leq n+r$ and $1\leq j_1<j_2<\cdots <j_{k+2}\leq n+r$. Let $s\in [k+2]$ be an index such that $j_1=i_1,\dots, j_{s-1}=i_{s-1},$ and $j_{s}<i_{s}$. Such an $s$ exists since $g>_{\lex} f$, and in particular we have $j_1\leq i_1$. Since $\left|\supp(g)\setminus \supp(f)\right|\geq 2$, we have $s\neq k+2$. We will prove the existence of $i$ and $j$ when we know the types of $f$ and $g$.

\begin{lemma}\label{lem:f-type-1}
    Assume that $f$ is of type 1. Then there exist $i,j\in [r+n]$ such that $i>j$, $j\in \supp(g)\setminus \supp(f)$, and $x_j\frac{f}{x_i}$ is of type 1.
\end{lemma}

\begin{proof}
    Since $f$ is of type 1, we have $i_{k+2} \in [r+1,\ r+n-4]$ and either of the following holds:
    \begin{enumerate}[label=(\roman*)]
        \item $\supp(f) \cap [r] \neq \emptyset$;
        \item $[\min Q_f, \ i_{k+2}]\setminus Q_f$.
    \end{enumerate} 
    It suffices to show that $f'\coloneqq x_{j_s} \frac{f}{x_{i_s}}$ is of type 1. Indeed, we first remark that $\max(f')=\max(f)=i_{k+2}\in [r+1,\ r+n-4]$.
    
    If $i_1\in [r]$, then $j_1\in [r]$ as well since $j_1\leq i_1$. 
    Thus, $\supp(f') \cap [r]\neq \emptyset$, and hence $f'$ is of type 1, as desired.
    
    Now we can assume that $i_1>r$. 
    Then $\supp(f) \cap [r] = \emptyset$. 
    In other words, (i) does not hold, and hence (ii) does, i.e., 
    $Q_f\subsetneq [\min Q_f,\ i_{k+2}]$. 
    If $j_s\in [r]$, then $\supp(f') \cap [r] \neq \emptyset$, and hence $f'$ is of type 1, as desired. 
    On the other hand, if $j_s>r$, then by construction, we have $\min Q_{f'}\leq \min Q_f$. 
    Hence  $i_s\in [\min Q_{f'},\ i_{k+2}] \setminus Q_{f'}$, which implies that $[\min Q_{f'},\ \max Q_{f'}] \setminus Q_{f'} \neq \emptyset$. 
    Therefore, $f'$ is of type 1, as desired.
\end{proof}

\begin{lemma}\label{lem:f-type-2}
    Assume that $f$ is of type 2. Then, there exist $i,j\in [r+n]$ such that $i>j$, $j\in \supp(g)\setminus \supp(f)$, and $x_j\frac{f}{x_i}$ is of type 2.
\end{lemma}
\begin{proof}
    Since $f$ is of type $2$, we have $i_{k+2}=r+n-3$. 
    Setting $f'\coloneqq x_{j_s} \frac{f}{x_{i_s}}$, it is clear that $\max(f')=i_{k+2}=r+n-3$ since $s\neq k+2$. Therefore $f'$ is of type 2, as desired.
\end{proof}

\begin{lemma}\label{lem:f-type-3A-or-3B}
    Assume that $f$ is of type 3A or 3B and $g$ is of type 1, 2, 3A, 3B, 3C, 4, or 5. Then there exist $i,j\in [r+n]$ such that $i>j$, $j\in \supp(g)\setminus \supp(f)$, and $x_j\frac{f}{x_i}$ is of type 1, 2, 3A, or 3B.
\end{lemma}
\begin{proof}
    Since $f$ is of type 3A or 3B, we have $i_{k+2}=r+ n - 2$ and either of the following holds:
        \begin{enumerate}[label=(\roman*)]
            \item $[\min Q_f,\ \secondmax(f)] = [\min Q_f, \ i_{k+1}]\setminus Q_f\neq \emptyset$;
            \item $i_{k+1}\in [r+1,\ r+n-5]$.
        \end{enumerate}
    If $j_s\in [r]$, then since $\max(x_{j_s} \frac{f}{x_{r+n-2}})\in [r+1,\ r+n-3]$ and $\supp(x_{j_s} \frac{f}{x_{r+n-2}})\cap [r]\supset \{j_s\}$, we have that $x_{j_s} \frac{f}{x_{r+n-2}}$ is of type 1 or 2. 
    Thus, $i=r+n-2$ and $j=j_s$ are what we need, as required. 
    Now we can assume that $j_s>r$. 
    Since $s<k+2$, we either have $s=k+1$, or $s<k+1$. 
    We will consider these two cases separately.

    Suppose that $s<k+1$. It suffices to show that $f'\coloneqq x_{j_s} \frac{f}{x_{i_s}}$ is of type 3A or 3B. Indeed, we have $\max(f')=\max(f)=r+n-2$ and $\secondmax(f')=\secondmax(f)=i_{k+1}$. Thus if (ii) holds, then $f'$ is of type 3B, as desired. Now we can assume that (i) holds, i.e., $i_{k+1}\in \{r+n-3,\ r+n-4\}$ and $ [\min Q_f, \ i_{k+1}]\setminus Q_f$. Since $j_s>r$, we have $\min Q_{f'}\leq \min Q_f$. Thus $i_s\in [\min Q_{f'}, \ i_{k+1}]\setminus Q_{f'}$, which implies that $[\min Q_{f'}, \ \secondmax(f')]\setminus Q_{f'} \neq \emptyset$. Therefore, $f'$ is of type 3A, as desired.
    
    Now we can assume that $s=k+1$, i.e., $\supp(g)\setminus \{j_{k+1},j_{k+2}\}= \supp(f)\setminus \{i_{k+1},i_{k+2}\}$. Since 
    $\left|\supp(g)\setminus \supp(f)\right|\geq 2$, we conclude that $j_{k+1},j_{k+2}, i_{k+1},i_{k+2}$ are four  distinct integers that are not in $\supp(f)\cup \supp(g)$.  Again it suffices to show that $f'\coloneqq x_{j_s} \frac{f}{x_{i_s}}$ is of type 3A or 3B. 
    We observe that $i_{s-1}=j_{s-1}<j_s$, and hence $\secondmax(f')= j_s$. If $j_s\in [r+1,\ r+n-5]$, then $f'$ is of type 3B, as desired. Since we already assumed that $r<j_s<i_s<i_{k+2}=r+n-2$, we can now suppose that $j_s\in \{r+n-3,\ r+n-4\}$. Since $j_s<i_s<i_{k+2}=r+n-2$, this forces $j_s=r+n-4$ and $i_s=r+n-3$. We have $r+n-4=j_s < j_{k+2}\leq r+n$, and $j_{k+2}\notin \{i_{k+1},i_{k+2}\}=\{r+n-3,\ r+n-2\}$. Hence $j_{k+2}\in \{r+n-1,\ r+n\}$. In particular, $g$ must be of type 4 or 5. Since $r+n-3\notin \supp(g)$, the monomial $g$ is of type $4$, i.e., $j_{k+2}=r+n-1$, $r+1\in \supp(g)$, and $r+2\notin \supp(g)$. In particular, we have $r+2<j_s= r+n-4$, and thus $r+1\in \supp(f')$ and $r+2\notin \supp(f')$ as well. Thus $r+2\in [r+1,r+n-4]=[\min Q_{f'},\ \secondmax(f')]$ and $r+2\notin Q_{f'}$. This implies that $[\min Q_{f'},\ \secondmax(f')]\setminus Q_{f'}\neq \emptyset$. Therefore, $f'$ is of type 3A, as desired.
\end{proof}

\begin{lemma}\label{lem:f-type-3C}
    Assume that $f$ is of type 3C. Then there exist $i,j\in [r+n]$ such that $i>j$, $j\in \supp(g)\setminus \supp(f)$, and $x_j\frac{f}{x_i}$ is of type 3C.
\end{lemma}

\begin{proof}
    It suffices to show that $x_{j_s}\frac{f}{x_{i_s}}$ is of type 3C. Indeed, since $f$ is type 3C, we have $\max(f)=r+n-2$ and $\supp(f)\cap [r]\neq \emptyset$. It is clear that $\max(x_{j_s}\frac{f}{x_{i_s}})=r+n-2$. If $i_s\notin[r]$, then  $\supp(x_{j_s}\frac{f}{x_{i_s}}) \cap [r] \neq \emptyset$ since it contains $\supp(f)\cap [r]\neq \emptyset$. On the other hand, if $i_s\in [r]$, then $j_s\in [r]$ as well. In either case, we have $\supp(x_{j_s}\frac{f}{x_{i_s}}) \cap [r] \neq \emptyset$. Thus $x_{j_s}\frac{f}{x_{i_s}}$ is of type 3C, as~desired.
\end{proof}

\begin{lemma}\label{lem:f-type-4-g-type-5}
    Assume that $k\in [n+r-4]\setminus \{n-4\}$. Assume that $f$ is of type 4, $g$ is of type 5, and $\supp(g)\setminus \supp(f) =\{r+2,\ r+n\}$. Then, there exist $i,j\in [r+n]$ such that $i>j$, $j\in \supp(g)\setminus \supp(f)$, and either of the following holds:
    \begin{enumerate}[label=(\roman*)]
        \item $k<n-4$ and $x_j\frac{f}{x_i}$ is of type 2 or 3A;
        \item $k>n-4$ and $x_j\frac{f}{x_i}$ is of type 2 or 3C.
    \end{enumerate}
\end{lemma}

\begin{proof}
    Set $f'\coloneqq x_{r+2}\frac{f}{x_{r+n-1}}$. It suffices to show that $f'$ is of the desired type. 
    Since $g$ is of type~5, we have $r+n-3\in \supp(g)$, 
    and since $r+n-3\notin \{r+2,\ r+n\}$, 
    we have $r+n-3\in \supp(f)$ as well. 
    Therefore, $r+n-3\in \supp(f')$, and in particular, this implies that $\max(f')$ is either $r+n-3$ or $r+n-2$. 
    If $\max(f') = r+n-3$, then $f'$ is of type 2, as desired.  Now we can assume that $\max(f') = r+n-2$. Then $\secondmax(f') = r+n-3$. Since $f$ is of type 4, we have $r+1\in \supp(f)$, and since $r+1\notin \{r+2,\ r+n\}$, we have $r+1\in \supp(f')$ as well. Thus $\min Q_{f'}=r+1$. We have two cases, depending on $k$:
    \begin{enumerate}[label=(\roman*)]
        \item If $k<n-4$, then we have
        \begin{align*}
            |Q_{f'} \setminus \{\max(f')\}| &\leq \left|\supp(f')\right| -1\\
            &= k+1 \\
            &< (n-4)+1 \\
            &= (r+n-3) - (r+1) + 1\\
            &= |[ \min Q_{f'}, \secondmax(f')]|.
        \end{align*}
        This, in particular, implies that $[ \min Q_{f'}, \secondmax(f')]\setminus Q_{f'} \neq \emptyset$. In other words, $f'$ is of type 3A, as desired.
        \item If $k>n-4$, then since $r+n-1, \ r+n \notin \supp(f')$, we have
        \[
        \left|\supp(f') \cap \left([r+n] \setminus [r]\right)\right| 
        \,\leq\, 
        \left|[r+n]\setminus \left([r] \cup \{r+n-1,\ r+n\}\right)\right| = n-2.
        \]
        On the other hand, we have $\left|\supp(f')\right|=k+2>n-2$. Therefore, $\supp(f')\nsubseteq \big( [r+n] \setminus [r]\big)$, or equivalently, $\supp(f')\cap [r]\neq \emptyset$. Thus $f'$ is of type 3C, as desired.\qedhere
    \end{enumerate}
\end{proof}

\begin{lemma}\label{lem:f-type-4-mostly}
    Assume that $f$ is of type 4 and $\supp(g)\setminus \supp(f) \neq \{r+2, \ r+n\}$. Then there exist $i,j\in [r+n]$ such that $i>j$, $j\in \supp(g)\setminus \supp(f)$, and $x_j\frac{f}{x_i}$ is of type 1, 2, 3A, 3B, or 4.
\end{lemma}

\begin{proof}
    Since $f$ is of type 4, we have $i_{k+2}=r+n-1$, $r+1\in \supp(f)$, and $r+2\notin \supp(f)$. 
    We have the following three cases:
    \begin{itemize}
        \item Suppose that $i_s\neq r+1$ and $j_s\neq r+2$. 
        Then $r+1\in \supp(x_{j_s}\frac{f}{i_s})$ and $r+2\notin \supp(x_{j_s}\frac{f}{i_s})$. 
        Moreover, we have $\max\left(x_{j_s}\frac{f}{i_s}\right) = r+n-1$. Thus $x_{j_s}\frac{f}{i_s}$ is of type 4. 
        Therefore $i=i_s$ and $j=j_s$ are what we need, as desired.
        \item Suppose that $i_s\neq r+1$ and $j_s= r+2$. 
        Since $\left|\supp(g)\setminus \supp(f)\right|\geq 2$, we can pick 
        $j_t\coloneqq \min \left(\supp(g)\setminus (\supp(f) \cup \{j_s\})\right)$. In particular, $j_t > j_s =r+2$. 
        
        If $j_t<r+n-1$, then it suffices to show that $f'\coloneqq x_{j_t}\frac{f}{x_{r+n-1}}$ is of type 1, 2, 3A or 3B. It is clear that $r+1=\min Q_{f'}$ and $r+2<j_t\leq \max(f')<r+n-1$. We then have the following three subcases:
        \begin{itemize}
            \item Suppose that $\max(f') \in [r+3,\ r+n-4]$. Then it is clear that 
            \[
            r+2\in [r+1,\max(f') ] \setminus \supp(f'),
            \] 
            which implies that
            \[
            [\min Q_{f'},\ \max(f')]\setminus Q_{f'}\neq \emptyset.
            \]
            Thus $f'$ is of type 1, as desired.
            \item  Suppose that $\max(f') = r+n-3$. Then $f'$ is of type 2, as desired. 
            \item Suppose that $\max(f') =r+n-2$. Then $f'$ is of type 3B, as desired, if $\secondmax(f')\in [r+1,\ r+n-5]$. On the other hand, if $\secondmax(f')\in \{r+n-3,\ r+n-4\}$, then since 
            \[
            r+2\in [r+1,\secondmax(f')) ]\setminus \supp(f'),
            \] 
            we have 
            \[
            [\min Q_{f'},\secondmax(f')]\setminus Q_{f'} \neq \emptyset.
            \]
            Thus $f'$ is of type 3A, as desired.
    \end{itemize}
        Since $j_t\notin \supp(f)$, we have $j_t\neq r+n-1$. Hence we can now assume that $j_t=r+n$. By the definition of $j_t$, the condition $j_t=r+n$ implies that $\supp(g)\setminus \supp(f) = \{j_s,j_t\}= \{r+2,\ r+n\}$, contradicting the hypothesis.
        
        \item Suppose that $i_s=r+1$. Since $j_s<i_s< i_{s+1}$, it suffices to show that $f''\coloneqq x_{j_s}\frac{f}{x_{i_{s+1}}}$ is of type 1 or 4. Indeed, it is clear from the construction that $r+1\in \supp(f'')$, and $r+2\notin \supp(f'')$. If $i_{s+1}<i_{k+2}$, then this guarantees that $\max(f'')= i_{k+2}= r+n-1$, and hence $f''$ is of type 4, as desired. Now, we can assume that $i_{s+1}=i_{k+2}$. This means that $Q_f=\{r+1,\ r+n\}$, and hence $Q_{f''}=\{r+1\}$. Moreover, we have $j_s\in \supp(f'')\cap [r]$. Thus $f''$ is of type 1, as desired. \qedhere
    \end{itemize}
\end{proof}

\begin{lemma}\label{lem:f-type-5}
    Assume that $f$ is of type 5. Then there exist $i,j\in [r+n]$ such that $i>j$, $j\in \supp(g)\setminus \supp(f)$, and $x_j\frac{f}{x_i}$ is of type 2.
\end{lemma}
    
\begin{proof}
    Since $f$ is of type 5, we have $i_{k+2}=r+n$ and $i_{k+1}=r+n-3$. It suffices to show that $f'\coloneqq x_{j_s}\frac{f}{x_{r+n}}$ is of type 2. Indeed, we have $j_s<i_s\leq i_{k+1}=r+n-3$. Thus $\max(f')= r+n-3$. Thus, $f'$ is of type 2, as claimed.
\end{proof}

We are now in a position to determine exactly all the values of $k$ for which $\HS_k(\CH_{n,r})$ and $\HS_k(\ACH_{n,r})$ have linear quotients. We start with $\HS_k(\ACH_{n,r})$. We remark that due to Lemma~\ref{lem:f-type-4-g-type-5}, monomials of type 3C may come into play when $k>n-4$. However, monomials of type 3C are not minimal generators of $\HS_k(\ACH_{n,r})$ by Proposition~\ref{prop:gens-TH}. Thus, it is of no surprise that the next result only holds for $k< n-4$.

\begin{proposition}\label{prop:TH-LQ}
    For each $k\in [2, \ n-4)$, the ideal $\HS_k(\CH_{n,r}^c)$ has linear quotients.
\end{proposition}
\begin{proof}
    By Lemma~\ref{lem:LQ-lex} and Proposition~\ref{prop:gens-CH}, it suffices to show that for any $f,g$ where 
    \begin{align*}
        &\left|\supp(f)\right|=\left|\supp(g)\right|=k+2,\\
        &Q_f\neq \emptyset, \ Q_g\neq \emptyset,\\
        &\begin{multlined}
            \text{$f$ and $g$ are of type 1, 2, 3A, 3B, 4, or 5}\\
            \text{(we note that $f$ and $g$ do not need to be of the same type)},
        \end{multlined}
    \end{align*}
    then there exist $i,j\in [r+n]$ such that $i>j$, $j\in \supp(g)\setminus \supp(f)$, and $x_j\frac{f}{x_i}$ is also of type 1, 2, 3A, 3B, 4, or 5. Indeed, this follows from Lemmas~\ref{lem:f-type-1}, \ref{lem:f-type-2}, \ref{lem:f-type-3A-or-3B}, \ref{lem:f-type-4-g-type-5}, \ref{lem:f-type-4-mostly}, and \ref{lem:f-type-5}, as desired.
\end{proof}

\begin{proposition}\label{prop:TH-no-LQ}
    For each $n-4\leq k\leq n+r-4$ and $r\geq 0$, the ideal $\HS_k(\ACH_{n,r}^c)$ does not have linear quotients.
\end{proposition}

\begin{proof}
    As $\ACH_{n,r}^c$ contains $\ACH_{n,r'}^c$ as an induced subgraph whenever $r>r'$, by Lemma~\ref{lem:linear-quotients-induced}, it suffices to assume that $r=k-n+4\geq 0$ and show that $\HS_k(\ACH_{n,r}^c)$ does not have linear quotients. Consider a minimal monomial generator $f=x_{i_1}x_{i_2}\cdots x_{i_{r+n-2}}$ of $\HS_k(\ACH_{n,r}^c)$, where $1\leq i_1<i_2<\cdots<i_{r+n-2}\leq n+r$. We note that $i_{r+n-2} \geq r+n-2$.  We will derive all the possible options for $\supp(f)$. 
    
    If $i_{r+n-2}\leq r+n-2$, then we must have $i_{r+n-2}=r+n-2$ and
    \[
    \supp(f) = [r+n-2].
    \]
    However, since $i_{r+n-2}= i_{k+2}= r+n-2$, in this case $f=x_1x_2\cdots x_{r+n-2}$ is not a minimal monomial generator of $\HS_k(\ACH_{n,r}^c)$ by Proposition~\ref{prop:gens-TH}.   

    If $i_{r+n-2}=r+n-1$, then $f=\frac{\prod_{i=1}^{r+n-1} x_i}{x_a}$ for some $a\in [r+n-3]$. By Proposition~\ref{prop:gens-TH}, $f$ is of type $4$. In particular, we have $a=r+2$ and $f= \frac{\prod_{i=1}^{r+n} x_i }{x_{r+2}x_{r+n}}$.

    Finally, we can now assume that $i_{r+n-2}=r+n$. By Proposition~\ref{prop:gens-TH} again, $f$ is of type 5. Since $\left|\supp(f)\right|=r+n-2$, we have $f=\frac{\prod_{i=1}^{k+4} x_i }{x_{r+n-2}x_{r+n-1}}$.
    
    Thus $\HS_k(\ACH_{n,k-n+4}^c)=\big( \frac{\prod_{i=1}^{r+n} x_i }{x_{r+n-2}x_{r+n-1}}, \ \frac{\prod_{i=1}^{r+n} x_i }{x_{r+2}x_{r+n}} \big)$, which does not have linear quotients by Lemma~\ref{lem:no-LQ-example}, as desired.
\end{proof}

\begin{proof}[Proof of Theorem~\ref{thm:HLQ-TH-graph}] The theorem follows from Propositions~\ref{prop:TH-LQ}, \ref{prop:TH-no-LQ}, and \ref{prop:proj-CH-TH}.
\end{proof}

As $\ACH_{n,0}^c=\HH_n^c$, we have the following.

\begin{theorem}\label{thm:HLQ-H}
    For each $n\geq 6$, we have 
    \begin{enumerate}[label=(\roman*)]
        \item $\HS_k(\HH_{n}^c)\neq 0$ has linear quotients if $2\leq k< n-4$;
        \item $\HS_{n-4}(\HH_{n}^c)\neq 0$ does not have linear quotients; 
        \item $\HS_k(\HH_{n}^c)=0$ if $n-4<k$.
    \end{enumerate}
\end{theorem}

Finally, we prove the results for $\HS_k(\CH_{n,r})$. 

\begin{proposition}\label{prop:CH-LQ}
    For each $k\in [2, \ r+n-4]\setminus \{n-4\}$, the ideal $\HS_k(\CH_{n,r}^c)$ has linear~quotients.
\end{proposition}
\begin{proof}
    By Lemma~\ref{lem:LQ-lex} and Proposition~\ref{prop:gens-CH}, it suffices to show that for any $f,g$ where 
    \begin{align*}
        &\left|\supp(f)\right|=\left|\supp(g)\right|=k+2,\\
        &Q_f\neq \emptyset, \ Q_g\neq \emptyset,\\
        &\begin{multlined}
            \text{$f$ and $g$ are of type 1, 2, 3A, 3B, 3C, 4, or 5}\\
            \text{(we note that $f$ and $g$ do not need to be of the same type)},
        \end{multlined}
    \end{align*}
    then there exist $i,j\in [r+n]$ such that $i>j$, $j\in \supp(g)\setminus \supp(f)$, and $x_j\frac{f}{x_i}$ is also of type 1, 2, 3A, 3B, 3C, 4, or 5. Indeed, this follows from Lemmas~\ref{lem:f-type-1}, \ref{lem:f-type-2}, \ref{lem:f-type-3A-or-3B}, \ref{lem:f-type-3C}, \ref{lem:f-type-4-g-type-5}, \ref{lem:f-type-4-mostly}, and \ref{lem:f-type-5}, as desired.
\end{proof}

\begin{proposition}\label{prop:CH-no-LQ}
    The ideal $\HS_{n-4}(\CH_{n,r}^c)$ does not have linear quotients.
\end{proposition}

\begin{proof}
    Since $\CH_{n,r}^c$ contains $\HH_{n}^c$ as an induced subgraph, the result follows from Lemma~\ref{lem:linear-quotients-induced} and Theorem~\ref{thm:HLQ-H}.
\end{proof}

\begin{proof}[Proof of Theorem~\ref{thm:HLQ-CH-graph}] The theorem follows from Propositions~\ref{prop:CH-LQ}, \ref{prop:CH-no-LQ}, and \ref{prop:proj-CH-TH}.
\end{proof}

\section{A discussion on the possible patterns} \label{sec:possible-patterns}

We give a partial answer to Conjecture~\ref{conj:main}.

\begin{theorem}\label{thm:conjecture-partial}
    If the edge ideal $I(G)$ has homological linear quotients, then $G$ is co-chordal and $\HH_n^c$-free for any $n\geq 6$. The converse holds if $G$ has at most 10 vertices.
\end{theorem}
\begin{proof}
    The first statement follows from Fr\"oberg's theorem \cite{Froberg}, Lemma~\ref{lem:linear-quotients-induced} and Theorem~\ref{thm:HLQ-H}, while the second from checking all graphs with at most 10 vertices, using the database \cite{HSI}.
\end{proof}

We came short of proving the conjecture in its full generality. The next question is rigidity, e.g., whether $\HS_k(G)$ having linear quotients implies that the same for either of $\HS_{k-1}(G)$ or $\HS_{k+1}(G)$. Some questions of this nature have a negative answer by the three main theorems in this paper. The next question for us, is to figure out which pattern is possible, and which is not. Following \cite{HSI}, we associate to a co-chordal graph $G$ (except for when $G$ has one edge) with a finite sequence of letters, denoted by $\{LQ_k(G)\}_{1\leq k \leq \pd I(G)}$, where
\[
LQ_k(G)=\begin{cases}
    T & \text{ if $\HS_k(G)\neq 0$ has linear quotients},\\
    F & \text{ if $\HS_k(G)\neq 0$ does not have linear quotients}.
\end{cases}
\]
The question becomes which sequence is possible for $\{LQ_k(G)\}_{1\leq k \leq \pd I(G)}$, where $G$ ranges over all co-chordal graphs. It is known that $LQ_1(G)=T$ for any co-chordal graph $G$ (\cite[Theorem 1.3]{FH23}). Our three main theorems imply that the following patterns are possible:
\begin{align*}
    T\underbrace{TTTTT}_{s}  \underbrace{FFFFFFFFF}_{t}, \\
    T \underbrace{FFFFF}_{s}  \underbrace{TTTTTTTTT}_{t}, \\
    T \underbrace{TTTT}_{s}F  \underbrace{TTTTTTTTT}_{t},
\end{align*}
for any pair of nonnegative integers $(s,t)$. We call these patterns \emph{predictable}. It turns out that these predictable patterns occur quite often, as can be seen in the following short summary of the database \cite{HSI}.

\begin{theorem}\label{thm:patterns}
    If $G$ is co-chordal, has more than one edge, and has at most 8 vertices, then $\{LQ_k(G)\}_{1\leq k \leq \pd I(G)}$ is of a predictable pattern. Tables~\ref{tab:nine} and~\ref{tab:ten} illustrate all the possible patterns for $\{LQ_k(G)\}_{1\leq k \leq \pd I(G)}$ when $G$ has 9 or 10 vertices, with the first column being the pattern, and the second column being the number of graphs with that pattern.
\end{theorem}
    
\begin{table}[h]
    \begin{tabular}{|l|c|}
        \hline 
        \emph{Pattern} & \emph{Count} \\
        \hline
        \hline
        TTFFTT & 69 \\ \hline
        TTFFT & 44 \\ \hline
        TTFFTTT & 10 \\ \hline
        TFTFF & 1 \\ \hline
        predictable & 12281 \\ \hline
    \end{tabular}
    \caption{Patterns when $G$ is co-chordal with 9 vertices.}
    \label{tab:nine}
\end{table}

\begin{table}[h]
    \begin{tabular}{|l|c|}
        \hline
        \emph{Pattern} & \emph{Count} \\ 
        \hline \hline
        TTFFFTT & 960 \\ \hline
        TTFFTTT & 885 \\ \hline
        TTFFFT & 551 \\ \hline
        TTFFTT & 263 \\ \hline
        TTFFTTTT & 144 \\ \hline
        TTFFFTTT & 123 \\ \hline
    \end{tabular}
    \begin{tabular}{|l|c|}
        \multicolumn{2}{c}{} \\ \hline
        TTTFFTT & 83 \\ \hline
        TTTFFT & 48 \\ \hline
        TFTFFF & 15 \\ \hline
        TTFFT & 14 \\ \hline
        TTTFFTTT & 10 \\ \hline
        TFTFFTT & 9 \\ \hline
    \end{tabular}
    \begin{tabular}{|l|c|}
        \multicolumn{2}{c}{} \\
        \hline
        TFTFFT & 5 \\ \hline
        TFTFFTTT & 1 \\ \hline
        TTFTFF & 1 \\ \hline
        TFTTFF & 1 \\ \hline
        TFTFF & 1 \\ \hline
        predictable & 109120 \\ \hline
    \end{tabular}
    \caption{Patterns when $G$ is co-chordal with 10 vertices.}
    \label{tab:ten}
\end{table}

Obviously, some patterns are impossible, e.g., those start with F, as it would contradict the known result that $LQ_1(G)=T$ for any co-chordal graph $G$. This raises a question: is there any impossible pattern that starts with T for co-chordal graphs? We present such a pattern.

\begin{theorem}\label{thm:impossible-pattern}
    Let $G$ be a co-chordal graph. Then the sequence $\{LQ_k(G)\}_{1\leq k \leq \pd I(G)}$ cannot be TFTF.
\end{theorem}
\begin{proof}
    Suppose, for the sake of contradiction, that there exists a co-chordal graph $G$ such that $\{LQ_k(G)\}_{1\leq k \leq \pd I(G)}$ is the pattern TFTF. This implies that $\pd I(G)= 4$ and, by Theorem~\ref{thm:patterns}, that $G$ has more than 10 vertices. 

    Since $G$ is co-chordal and has more than 10 vertices, it contains a co-chordal induced subgraph $G'$ with 11 vertices. Then $\pd I(G')\leq \pd I(G) = 4$, and $LQ_3(G')=T$ by Lemma~\ref{lem:linear-quotients-induced}. It suffices to show that such a $G'$ does not exist.
    
    Indeed, since $G'$ is co-chordal and has more than 10 vertices, it contains a co-chordal induced subgraph $G''$ with 10 vertices. Then $\pd I(G'')\leq \pd I(G') = 4$, and $LQ_3(G'')=T$ by Lemma~\ref{lem:linear-quotients-induced}. According to the database \cite{HSI}, there are exactly three such graphs. Then we can construct all possible choices for $G'$. There are $1233$ such graphs, and we verified that either $\pd I(G')>4$ or that $\{LQ_k(G')\}_{1\leq k \leq \pd I(G')}$ has the TFFF pattern. This contradicts the fact that $LQ_3(G')=T$, as desired.  
\end{proof}

\bibliographystyle{amsplain}
\bibliography{refs}

\providecommand{\bysame}{\leavevmode\hbox to3em{\hrulefill}\thinspace}
\providecommand{\MR}{\relax\ifhmode\unskip\space\fi MR }
\providecommand{\MRhref}[2]{%
  \href{http://www.ams.org/mathscinet-getitem?mr=#1}{#2}
}
\providecommand{\href}[2]{#2}
\begin{thebibliography}{10}

\bibitem{MR3503700}
Karim~A. Adiprasito, Eran Nevo, and Jose~A. Samper, \emph{Higher chordality:
  from graphs to complexes}, Proc. Amer. Math. Soc. \textbf{144} (2016), no.~8,
  3317--3329. \MR{3503700}

\bibitem{AlFa2018}
A.~Alilooee and S.~Faridi, \emph{Graded betti numbers of path ideals of cycles
  and lines}, Journal of Algebra and its Applications \textbf{17} (2018).

\bibitem{AlFa2015}
Ali Alilooee and Sara Faridi, \emph{On the resolution of path ideals of
  cycles}, Comm. Algebra \textbf{43} (2015), no.~12, 5413--5433. \MR{3395713}

\bibitem{Banerjee2015}
Arindam Banerjee, \emph{The regularity of powers of edge ideals}, J. Algebraic
  Combin. \textbf{41} (2015), no.~2, 303--321. \MR{3306074}

\bibitem{Bayati23}
Shamila Bayati, \emph{A quasi-additive property of homological shift ideals},
  Bull. Malays. Math. Sci. Soc. \textbf{46} (2023), no.~3, Paper No. 111, 17.
  \MR{4581528}

\bibitem{BPS98}
Dave Bayer, Irena Peeva, and Bernd Sturmfels, \emph{Monomial resolutions},
  Math. Res. Lett. \textbf{5} (1998), no. 1--2, 31--46.

\bibitem{ForestsCycles}
Selvi Beyarslan, Huy~T\`ai H\`a, and Tr\^an~Nam Trung, \emph{Regularity of
  powers of forests and cycles}, J. Algebraic Combin. \textbf{42} (2015),
  no.~4, 1077--1095. \MR{3417259}

\bibitem{MR4052306}
Mina Bigdeli and Sara Faridi, \emph{Chordality, {$d$}-collapsibility, and
  componentwise linear ideals}, J. Combin. Theory Ser. A \textbf{172} (2020),
  105204, 33. \MR{4052306}

\bibitem{MR3641976}
Mina Bigdeli, J\"urgen Herzog, Ali~Akbar Yazdan~Pour, and Rashid Zaare-Nahandi,
  \emph{Simplicial orders and chordality}, J. Algebraic Combin. \textbf{45}
  (2017), no.~4, 1021--1039. \MR{3641976}

\bibitem{HSI}
Trung Chau, Kanoy~Kumar Das, and Aryaman Maithani, \emph{Homological shift
  ideal database and documentaion}, Available at
  \url{https://aryamanmaithani.github.io/research_codes/HSI/}.

\bibitem{CHM24}
Trung Chau, Huy~T{\`a}i H{\`a}, and Aryaman Maithani, \emph{Minimal cellular
  resolutions of powers of graphs}, arXiv:2404.04380.

\bibitem{ConcaDeNegri1998}
Aldo Conca and Emanuela De~Negri, \emph{{$M$}-sequences, graph ideals, and
  ladder ideals of linear type}, J. Algebra \textbf{211} (1999), no.~2,
  599--624. \MR{1666661}

\bibitem{MR3092695}
E.~Connon and S.~Faridi, \emph{Chorded complexes and a necessary condition for
  a monomial ideal to have a linear resolution}, J. Combin. Theory Ser. A
  \textbf{120} (2013), no.~7, 1714--1731. \MR{3092695}

\bibitem{CF24}
Marilena Crupi and Antonino Ficarra, \emph{Very well-covered graphs via the
  {R}ees algebra}, Mediterr. J. Math. \textbf{21} (2024), no.~4, Paper No. 135,
  17. \MR{4754372}

\bibitem{Dirac1961OnRC}
G.~A. Dirac, \emph{On rigid circuit graphs}, Abhandlungen aus dem
  Mathematischen Seminar der Universit{\"a}t Hamburg \textbf{25} (1961),
  71--76.

\bibitem{MR2603461}
Eric Emtander, \emph{A class of hypergraphs that generalizes chordal graphs},
  Math. Scand. \textbf{106} (2010), no.~1, 50--66. \MR{2603461}

\bibitem{FH23}
Antonino Ficarra and J\"urgen Herzog, \emph{Dirac's theorem and multigraded
  syzygies}, Mediterr. J. Math. \textbf{20} (2023), no.~3, Paper No. 134, 18.
  \MR{4549911}

\bibitem{FrancisciHa2008}
Christopher~A. Francisco and Huy~T\`ai H\`a, \emph{Whiskers and sequentially
  {C}ohen-{M}acaulay graphs}, J. Combin. Theory Ser. A \textbf{115} (2008),
  no.~2, 304--316. \MR{2382518}

\bibitem{FranciscoVanTuyl2007}
Christopher~A. Francisco and Adam Van~Tuyl, \emph{Sequentially
  {C}ohen-{M}acaulay edge ideals}, Proc. Amer. Math. Soc. \textbf{135} (2007),
  no.~8, 2327--2337. \MR{2302553}

\bibitem{Froberg}
Ralf Fr{\"o}berg, \emph{On stanley-reisner rings}, Banach Center Publications
  \textbf{26} (1990), 57--70.

\bibitem{HerzogHibiBook}
J\"urgen Herzog and Takayuki Hibi, \emph{Monomial ideals}, Graduate Texts in
  Mathematics, vol. 260, Springer-Verlag London, Ltd., London, 2011.
  \MR{2724673}

\bibitem{HHZ04Dirac}
J\"urgen Herzog, Takayuki Hibi, and Xinxian Zheng, \emph{Dirac's theorem on
  chordal graphs and {A}lexander duality}, European J. Combin. \textbf{25}
  (2004), no.~7, 949--960. \MR{2083448}

\bibitem{HHZ2004}
\bysame, \emph{Monomial ideals whose powers have a linear resolution}, Math.
  Scand. \textbf{95} (2004), no.~1, 23--32. \MR{2091479}

\bibitem{HMRZ23}
J\"urgen Herzog, Somayeh Moradi, Masoomeh Rahimbeigi, and Guangjun Zhu,
  \emph{Some homological properties of {B}orel type ideals}, Comm. Algebra
  \textbf{51} (2023), no.~4, 1517--1531. \MR{4552908}

\bibitem{HMRG20}
Jürgen Herzog, Somayeh Moradi, Masoomeh Rahimbeigi, and Zhu Guangjun,
  \emph{Homological shift ideals}, Collect. Math. (2021), 157--174.

\bibitem{MoreyVillarrealSurvey2012}
Susan Morey and Rafael~H. Villarreal, \emph{Edge ideals: algebraic and
  combinatorial properties}, Progress in commutative algebra 1, de Gruyter,
  Berlin, 2012, pp.~85--126. \MR{2932582}

\bibitem{MR3979275}
Ashkan Nikseresht, \emph{Chordality of clutters with vertex decomposable dual
  and ascent of clutters}, J. Combin. Theory Ser. A \textbf{168} (2019),
  318--337. \MR{3979275}

\bibitem{SVV1994}
Aron Simis, Wolmer~V. Vasconcelos, and Rafael~H. Villarreal, \emph{On the ideal
  theory of graphs}, J. Algebra \textbf{167} (1994), no.~2, 389--416.
  \MR{1283294}

\bibitem{TBR24}
Nadia Taghipour, Shamila Bayati, and Farhad Rahmati, \emph{Homological linear
  quotients and edge ideals of graphs}, Bull. Aust. Math. Soc. \textbf{110}
  (2024), no.~2, 291--302. \MR{4803157}

\bibitem{VanTuylVillarreal2008}
Adam Van~Tuyl and Rafael~H. Villarreal, \emph{Shellable graphs and sequentially
  {C}ohen-{M}acaulay bipartite graphs}, J. Combin. Theory Ser. A \textbf{115}
  (2008), no.~5, 799--814. \MR{2417022}

\bibitem{Villarreal1990}
Rafael~H. Villarreal, \emph{Cohen-{M}acaulay graphs}, Manuscripta Math.
  \textbf{66} (1990), no.~3, 277--293. \MR{1031197}

\end{thebibliography}
\end{document}